\documentclass[12pt]{amsart}
\usepackage[margin=2.0cm]{geometry}

\usepackage{amsfonts}
\usepackage{mathrsfs}
\usepackage{cite}
\usepackage{enumerate}

\newcommand{\R}{{\mathbb R}}

\newtheorem{theorem}{Theorem}[section]
\newtheorem{lemma}[theorem]{Lemma}
\newtheorem{prop}[theorem]{Proposition}
\newtheorem{corollary}[theorem]{Corollary}
\newtheorem{conj}[theorem]{Conjecture}

\newtheorem{example}[theorem]{Example}

\newcommand{\inte}{{\operatorname{int}}}
\newcommand{\supp}{{\operatorname{supp}}}

\newcommand{\sfe}{{\mathbb S}^{n-1}}

\title{The $L_p$-Minkowski problem for $-n< p<1$}
\author[G.~Bianchi, K.J.~B\"or\"oczky, A.~Colesanti, D.~Yang]{Gabriele Bianchi, K\'aroly J. B\"or\"oczky, Andrea Colesanti, Deane Yang}
\address{Dipartimento di Matematica e Informatica ``U. Dini", Universit\`a di Firenze, Viale Morgagni 67/A, Firenze, Italy I-50134} \email{gabriele.bianchi@unifi.it}
\address{Alfr\'ed R\'enyi Institute of Mathematics, Hungarian Academy
  of Sciences, Reltanoda u. 13-15, H-1053 Budapest, Hungary, and
Department of Mathematics, Central European University, Nador u 9, H-1051, Budapest, Hungary} 
\email{boroczky.karoly.j@renyi.mta.hu}
\address{Dipartimento di Matematica e Informatica ``U. Dini", Universit\`a di Firenze, Viale Morgagni 67/A, Firenze, Italy I-50134} \email{andrea.colesanti@unifi.it}
\address{Department of Mathematics, New York University Tandon School of Engineering, 6 Metrotech Center, Brooklyn, NY 11201 U.S.A.}
\email{deane.yang@nyu.edu}
\subjclass[2010]{Primary: 52A38,35J96}
\keywords{$L_{p}$ Minkowski problem, Monge-\'Ampere equation}
\thanks{First and third authors are supported in part by the Gruppo Nazionale per l'Analisi Matematica, la Probabilit\`a e le loro Applicazioni (GNAMPA) of the Istituto Nazionale di Alta Matematica (INdAM). Second author is supported in part by
NKFIH grants 116451, 121649 and 129630.}

\begin{document}
\maketitle

\begin{abstract}
Chou and Wang's existence result for the $L_p$-Minkowski problem  on ${\mathbb S}^{n-1}$ for $p\in(-n,1)$  and an absolutely continuous 
measure  is discussed and extended to more general measures. In particular, we provide an almost optimal sufficient condition for the 
case $p\in(0,1)$.
\end{abstract}

\section{Introduction}

The setting for this paper is the $n$-dimensional Euclidean space $\mathbb{R}^{n}$. A \emph{convex body} $K$ in $\mathbb{R}^{n}$ is a 
compact convex set that has non-empty interior. For any $x\in\partial K$, $\nu_K(x)$ (``the Gau{\ss} map'') is the family of all unit exterior 
normal vectors at $x$; in particular $\nu_K(x)$ consists of a unique vector for $\mathcal{H}^{n-1}$ almost all $x\in\partial K$ 
(see, {\em e.g.}, Schneider \cite{SCH}), where
$\mathcal{H}^{n-1}$ stands for the ($n-1$)-dimensional Hausdorff measure. 

The \emph{surface area measure} $S_{K}$ of $K$ is a Borel measure on the unit sphere ${\mathbb S}^{n-1}$ of ${\mathbb R}^n$, defined, for a Borel 
set $\omega\subset {\mathbb S}^{n-1}$ by
$$
S_{K}(\omega)=\mathcal{H}^{n-1}\left(\nu_K^{-1}(\omega)\right)=
\mathcal{H}^{n-1}\left(\{x\in\partial K:\,\nu_K(x)\cap\omega\neq\emptyset\}\right)
$$
(see, {\em e.g.}, Schneider \cite{SCH}).

As one of the cornerstones of the classical Brunn-Minkowski theory, the Minkowski's existence theorem can be stated as follows  (see, {\em e.g.}, 
Schneider \cite{SCH}): If the Borel measure $\mu$ is not concentrated on a great subsphere of ${\mathbb S}^{n-1}$, then $\mu$ is the surface area measure of a convex 
body if and only if the following vector condition is verified 
$$
\int_{{\mathbb S}^{n-1}}ud\mu(u)=0.
$$
Moreover, the solution is unique up to translation. The regularity of the solution has been also well investigated, see {\em e.g.},
Lewy \cite{LE}, Nirenberg \cite{NIR}, Cheng and Yau \cite{CY}, Pogorelov \cite{POG}, and Caffarelli \cite{Caf90a,Caf90b}.

The surface area measure of a convex body has a clear geometric significance. In \cite{LUT},  Lutwak showed that there is an $L_{p}$ analogue 
of the surface area measure (known as the $L_{p}$-surface area measure). For a convex compact set $K$ in $\R^n$, let $h_K$ be its support 
function:
$$
h_K(u)=\max\{\langle x,u\rangle:\, x\in K\} \mbox{ \ \ for $u\in\R^d$},
$$
where $\langle\cdot,\cdot\rangle$ stands for the Euclidean scalar product.

Let ${\mathcal K}_0^n$ denote the family of convex bodies in $\R^n$ containing the origin $o$.  Note that if $K\in{\mathcal K}_0^n$,
then $h_K\ge0$. If $p\in\R$ and 
$K\in {\mathcal K}_0^n$, then the  $L_{p}$-surface area measure is defined by
$$
dS_{K,p}=h_K^{1-p}\,d S_K
$$
where for $p>1$ the right hand side is assumed to be a finite measure.
In particular, if $p=1$, then $S_{K,p}=S_K$, and if $p<1$ and $\omega\subset {\mathbb S}^{n-1}$ is a Borel set, then
$$
S_{K,p}(\omega)=\int_{x\in\nu_{K}^{-1}(\omega)}\langle x,\nu_{K}(x)\rangle^{1-p}d\mathcal{H}^{n-1}(x).
$$

In recent years, the $L_{p}$-surface area measure appeared in, {\em e.g.}, \cite{ADA,BG, CG, GM, H1, HP1, HP2, HENK, LU1, LU2, LR, LYZ1, LYZ2, LYZ3, LYZ6, LZ, NAO, NR, PAO, PW, ST3}. In \cite{LUT}, Lutwak posed the associated $L_{p}$-Minkowski problem for $p\geq 1$ which extends the classical Minkowski problem. 
In addition, the $L_p$-Minkowski problem for $p<1$ was publicized by a series of talks by Erwin Lutwak in the 1990's, and appeared in print in Chou and Wang \cite{CW} for the first time. \\

\textbf{$L_p$-Minkowski problem:} For $p\in\R$, what are the necessary and
sufficient conditions on a finite Borel measure $\mu$ on ${\mathbb S}^{n-1}$ in order 
that $\mu$ is the $L_p$-surface area measure of a convex body $K\in {\mathcal K}_0^n$?\\

Besides discrete measures, an important special class is that of Borel measures $\mu$ on ${\mathbb S}^{n-1}$ which have a density with respect to ${\mathcal H}^{n-1}$:
\begin{equation}
\label{densityfunction}
d\mu=f\,d{\mathcal H}^{n-1}
\end{equation}
for some non-negative measurable function $f$ on ${\mathbb S}^{n-1}$. 
If (\ref{densityfunction}) holds, then the $L_p$-Minkowski problem amounts to solving the Monge-Amp\`ere type equation
\begin{equation}
\label{MongeAmper}
h^{1-p}\det(\nabla^2h+h I)=f
\end{equation}
where $h$ is the unknown  non-negative (support) function on ${\mathbb S}^{n-1}$ to be found, $\nabla^2 h$ denotes the (covariant) Hessian matrix of $h$ with respect to an orthonormal frame on ${\mathbb S}^{n-1}$, and $I$ is the identity matrix. Recent extensions of the $L_p$-Minkowski problem are the $L_p$ dual Minkowski problem proposed by
Lutwak, Yang, Zhang \cite{LYZ18},  and the Orlicz Minkowski problem discussed by Haberl, Lutwak, Yang, Zhang \cite{HLYZ0}
(extending the case $p>1$, for even measures),  Huang, He \cite{HuH12} (extending the case $p>1$) and
Jian, Lu \cite{JLZ18+} (extending the case $0<p<1$). 

The case $p=1$, namely the classical Minkowski problem, was solved by Minkowski \cite{MIN} in the case of polytopes,
and in the general case by Alexandrov \cite{Ale38}, and Fenchel and Jessen \cite{FeJ38}.
The case $p>1$ and $p\neq n$ was solved by Chou and Wang \cite{CW},  Guan and Lin \cite{GL} and Hug, Lutwak, Yang, and Zhang \cite{HLYZ2}; 
Zhu \cite{Zhu17} investigated the dependence of the solution on $p$ for a given target measure.
We note that the solution is unique if $p>1$ and $p\neq n$, and unique up to translation if $p=1$. 
In addition, if $p>n$, then the origin lies in the interior of the solution $K$; however, if $1<p<n$, then
possibly the origin lies on the boundary of the solution $K$ even if (\ref{densityfunction}) holds for a positive continuous $f$.

The goal of this paper is to discuss the $L_p$-Minkowski problem for $p<1$. The case $p=0$ is the so called
logarithmic Minkowski problem see, {\em e.g.}, 
\cite{BLYZ, BLYZ2, BLYZ3, BoH16, LU1, LU2, LR, NAO, NR, PAO, ST1, ST2, ST3, Z2}.
Additional references regarding the $L_{p}$ Minkowski problem and Minkowski-type problems can be found in, {\em e.g.}, 
\cite{WC, CW, GG, GL, GM, H1, HL1, HLYZ0, HLYZ1, HMS, JI, KL, LWA, LUT, LO, LYZ5, MIN, ST1, ST2, Z3, Zhu15}. 
Applications of the solutions to the $L_{p}$ Minkowski problem can be found in, {\em e.g.}, \cite{AN1, AN2, KSC, Z, GH, LYZ4, CLYZ, HS1, HUS, Iva13, HS2, HSX, Wan12,Wan15}.

We note that if $p<1$, then non-congruent $n$-dimensional convex bodies may give rise to the same $L_p$-surface area measure, see 
 Chen, Li, and Zhu \cite{CSL01} for examples when $0<p<1$, Chen, Li, and Zhu \cite{CSL0} for examples when $p=0$
and  Chou and Wang \cite{CW} for examples when $p<0$.

If $0<p<1$, then the $L_p$-Minkowski problem is essentially solved by Chen, Li, and Zhu \cite{CSL01}.

\begin{theorem}[Chen, Li, and Zhu]
\label{theo01}
If $p\in(0,1)$, and $\mu$ is a finite Borel measure  on ${\mathbb S}^{n-1}$ not concentrated on a great subsphere, then
 $\mu$ 
is the $L_p$-surface area measure of a convex body  $K\in{\mathcal K}_{0}^n$. 
\end{theorem}

We believe that the following  property characterizes $L_p$-surface area measures for $p\in(0,1)$.

\begin{conj}
\label{conjdens01}
Let $p\in(0,1)$, and let $\mu$ be a non-trivial Borel measure  on ${\mathbb S}^{n-1}$. Then $\mu$ 
is the $L_p$-surface area measure of a convex body  $K\in{\mathcal K}_{0}^n$ if and only if 
$\supp\,\mu$ is not a pair of antipodal points. 
\end{conj}

Conjecture~\ref{conjdens01} is proved in the planar case $n=2$ independently by B\"or\"oczky and Trinh \cite{BoT} 
and Chen, Li,and Zhu \cite{CSL01}. 
Here we prove a slight extension of the result proved in \cite{CSL01}. We note that Lemma~\ref{01lpnotantipodal} of the present
paper implies that
$\supp\, S_{K,p}$ is not a pair of antipodal points for any convex body  $K\in{\mathcal K}_{0}^n$ and $p<1$.
For $X\subset \R^n$, its positive hull is
$$
{\rm pos}\,X=\left\{\sum_{i=1}^k\lambda_ix_i:\,\mbox{$\lambda_i\geq 0$, $x_i\in X$ and $k\geq 1$ integer}
\right\},
$$
which is closed if $X\subset {\mathbb S}^{n-1}$ is compact. We prove the following result. 

\begin{theorem}
\label{theodens01}
Let $p\in(0,1)$, let $\mu$ be a non-trivial finite Borel measure  on ${\mathbb S}^{n-1}$,
and let $L={\rm lin}\,\supp\,\mu$. If either $\supp\,\mu$ spans $\R^n$, or 
${\rm dim}\,L\leq n-1$ and ${\rm pos}\,\supp\,\mu\neq L$, then $\mu$ 
is the $L_p$-surface area measure of a convex body  $K\in{\mathcal K}_{0}^n$. In addition, if $\mu$ is invariant under 
a closed subgroup $G$ of $O(n)$ acting as the identity on $L^\bot$, then $K$ can be chosen to be invariant under $G$.
\end{theorem}

The  assumption in Theorem~\ref{theodens01} can be equivalently stated in term of the subset ${\rm conv}\left(\{o\}\cup\supp\,\mu\right)$ in $\R^n$ (here ${\rm conv} A$ denotes the convex hull of the set $A$). We require that either ${\rm conv}\left(\{o\}\cup\supp\,\mu\right)$ has non-empty interior or, if this is not the case, that ${\rm conv}\left(\{o\}\cup\supp\,\mu\right)$ does not contain $o$ in its relative interior.

The case $p=0$ concerns the cone volume measure. We say that a Borel measure $\mu$ on ${\mathbb S}^{n-1}$ satisfies the subspace concentration condition if for any non-trivial linear subspace $L$ we have
$$
\mu(L\cap {\mathbb S}^{n-1})\leq \frac{{\rm dim}\,L}n\,\mu({\mathbb S}^{n-1}),
$$
and equality holds if and only if there exists a complementary linear subspace $L'$ such that
 $\supp\,\mu\subset L\cup L'$.  B\"{o}r\"{o}czky, Lutwak, Yang, and Zhang \cite{BLYZ}
proved that even cone volume measures are characterized by the subspace concentration condition. The sufficiency part 
has been extended to all Borel measures on ${\mathbb S}^{n-1}$ by  Chen, Li, and Zhu \cite{CSL0}. The part of Theorem~\ref{theodens0} concerning the action of a 
 closed subgroup $G$ of $O(n)$ is not actually in \cite{CSL0} but could be verified easily using the methods of our paper.

\begin{theorem}[Chen, Li, Zhu]
\label{theodens0}
If  $\mu$ is a Borel measure  on ${\mathbb S}^{n-1}$  satisfying the subspace concentration condition, then $\mu$ 
is the $L_0$-surface area measure of a convex body  $K\in{\mathcal K}_{0}^n$. In addition, 
if $\mu$ is invariant under a closed subgroup $G$ of $O(n)$, then $K$ can be chosen to be invariant under $G$.
\end{theorem}

If $p=0$, then not even a conjecture is known concerning which properties may characterize $L_0$-surface area measures.
Note that B\"or\"oczky and  Heged\H{u}s \cite{BoH15}  characterized the restriction of an $L_0$-surface area measure
to a pair of antipodal points.

The main new result of this paper is the following statement regarding the case $p\in(-n,0)$.

\begin{theorem}
\label{theodens-n0}
If $p\in(-n,0)$, and $\mu$ is a non-trivial Borel measure  on ${\mathbb S}^{n-1}$  satisfying (\ref{densityfunction}) for a non-negative  function $f$ in $L_{\frac{n}{n+p}}(\sfe)$, then $\mu$ 
is the $L_p$-surface area measure of a convex body  $K\in{\mathcal K}_{0}^n$. 
In addition, if $\mu$ is invariant under a closed subgroup $G$ of $O(n)$, then $K$ can be chosen to be invariant under $G$.
\end{theorem}

It is not clear whether  the analogue of Theorem~\ref{theodens-n0} can be expected in the critical case $p=-n$. 
If $\partial K$ is $C^2_+$ and $o\in{\rm int}\,K$, then $L_{-n}$ surface area measure is
\begin{equation}
\label{SK-n}
dS_{K,-n}=\frac{h_K(u)^{n+1}}{\kappa(u)}\,d\mathcal{H}^{n-1},
\end{equation}
where $\kappa(u)$ is the Gaussian curvature of $\partial K$ at the point $x\in\partial K$ with $u\in\nu_K(x)$. 
Note that 
$\kappa_0(u)=\kappa(u)/h_K(u)^{n+1}$ is the so called  centro-affine curvature (see Ludwig \cite{LU2} or Stancu \cite{ST3}), which is equi-affine invariant in the following sense. 
For any $A\in {\rm SL}(n)$, if $\tilde{A}(u)=\frac{Au}{\|Au\|}$ is the corresponding projective transformation of ${\mathbb S}^{n-1}$, and $\tilde{\kappa}_0$ is the   centro-affine curvature function of $A^{-t}K$, then
$$
\tilde{\kappa}_0(\tilde{A}(u))=\kappa_0(u),\quad\forall\, u\in\sfe.
$$
In particular, Chou and Wang \cite{CW} proved the following formula for the $L_{-n}$ surface area measure.

\begin{prop}[Chou and Wang]
\label{theodensp-n}
Let $K\in{\mathcal K}_{0}^n$ be such that $o\in{\rm int}\,K$ and $\partial K$ is $C^3_+$, so that
$dS_{K,-n}=f\,d{\mathcal H}^{n-1}$ 
for a $C^1$ function $f$ according to \eqref{SK-n}. If $\mathcal{V}(\xi)=\xi_jA^{ij}\partial_i$ is a projective vector field on ${\mathbb S}^{n-1}$ for $A\in {\rm GL}(n)$, then
$$
\int_{{\mathbb S}^{n-1}}h^{-n}\,\mathcal{V}f\,d {\mathcal H}^{n-1}=0.
$$
\end{prop}

For the sake of completeness, we provide a proof of Proposition~\ref{theodensp-n} in Section~\ref{secCritical}.

We will prove Theorems~\ref{theodens01} and \ref{theodens-n0}  via an approximation argument based on
Theorem~\ref{theodens}, proved by Chou and Wang \cite{CW}. Of the latter, we will also provide a simplified and clarified argument.
Again, the part of Theorem~\ref{theodens} concerning the action of a 
 closed subgroup $G$ of $O(n)$ is not actually in \cite{CSL0} but could be verified easily using the methods of our paper.

\begin{theorem}[Chou and Wang]
\label{theodens}
If $p\in(-n,1)$, and $\mu$ is a Borel measure  on ${\mathbb S}^{n-1}$  satisfying (\ref{densityfunction}) where $f$ is bounded and $\inf_{u\in {\mathbb S}^{n-1}} f(u)>0$, then $\mu$ 
is the $L_p$-surface area measure of a convex body  $K\in{\mathcal K}_{0}^n$. 
In addition, if $\mu$ is invariant under the closed subgroup $G$ of $O(n)$, 
then $K$ can be chosen to be invariant under $G$, and $o\in{\rm int}\,K$ provided $p\in (-n,2-n]$.
\end{theorem}

\bigskip

\noindent 
{\bf Remark } Theorems~\ref{theodens01}, \ref{theodens0} and \ref{theodens-n0}
 show that Theorem~\ref{theodens} holds for any $p\in(-n,1)$ and non-negative bounded $f$
with $\int_{{\mathbb S}^{n-1}}f\,d\mathcal{H}^{n-1}>0$.\\

As already mentioned, if $p=0$, then B\"or\"oczky and Heged\H{u}s \cite{BoH15} provides some necessary condition on an $L_0$ surface area measure, more precisely, on the restriction of an $L_0$-surface area measure to pairs of antipodal points. Unfortunately, no necessary condition concerning $L_p$-surface area measures is known to us for the case $p<0$.

We conclude by mentioning the related paper by G.~Bianchi, K.~J.~B\"or\"oczky and A.~Colesanti \cite{BiBoCo} which deals with the strict convexity and the $C^1$ smoothness of the solution to the $L_p$ Minkowski problem when $p<1$ and $\mu$ satisfies~\eqref{densityfunction} for some function $f$ which is bounded from above and from below by positive constants.

\section{Preparation}

Let $\kappa_n$ be the volume of the $n$-dimensional unit Euclidean ball $B^n$, and let $\sigma(K)$ be 
the \emph{centroid} of a convex body $K$. 

\begin{lemma}
\label{centroid}
For a convex body $K$ in $\R^n$,
\begin{description}
\item[(i)] $ \frac{-1}n(x-\sigma(K))+\sigma(K)\in K$ for any $x\in K$;
\item[(ii)] (Blaschke-Santal\'o inequality)
$$
\int_{{\mathbb S}^{n-1}}\frac1{n(h_K(u)-\langle \sigma(K),u\rangle)^n}\,d{\mathcal H}^{n-1}(u)\leq\frac{\kappa_n^2}{V(K)}.
$$
\item[(iii)] If $\varrho>0$ is maximal and $R>0$ is minimal such that $\sigma(K)+\varrho\,B^n\subset K$ and 
$K\subset\sigma(K)+R\,B^n$, then
$$
V(K)\leq (n+1)\kappa_{n-1}\varrho R^{n-1}.
$$
\end{description}
\end{lemma}
\proof  In the case of the Blaschke-Santal\'o inequality, we note that if the origin is the centroid of $K$, then the left hand side of (ii) is the volume of the polar body $K^*$, and the origin is the Santal\'o point of $K^*$.
Therefore (i) and (ii) are well-known facts, see Lemma 2.3.3 and (10.28) in \cite{SCH}. 

For (iii), we assume that $\sigma(K)=o$. Let $x_0\in \varrho B^n\cap\partial K$, and let $H$ be the common tangent hyperplane to $K$ and $\varrho B^n$ at $x_0$. Since $-x/n\in K$ for any $x\in K$ as $\sigma(K)=o$, we deduce that $K$ lies between the parallel hyperplanes $H$ and $-nH$ whose distance is $(n+1)\varrho$. Note that $x_0$ is orthogonal to $H$. Now the projection of $K$ into $x_0^\bot$ is contained in $RB^n$, we conclude (iii).  Q.E.D.\\

For $v\in {\mathbb S}^{n-1}$ and $\alpha\in(0,\frac{\pi}2]$, let $\Omega(v,\alpha)$ be the family of all $u\in {\mathbb S}^{n-1}$ 
with $\angle(u,v)\leq\alpha$, where $\angle(u,v)$ is the (smaller) angle formed by $u$ and $v$, i.e. their geodesic distance on the unit sphere.
The following lemma is needed to show that with modified ``energy function'' $\varphi_\varepsilon$ (see next section), the optimal ``center'' 
is in the interior. 
 
\begin{lemma}
\label{oboundary}
Let $\varepsilon\in(0,\frac13]$, $R\geq 1$ and $q\geq n-1$; let $K\in{\mathcal K}_0^n$ with $o\in\partial K$ and 
${\rm diam}\,K\leq R$, and let $v$ be an exterior unit normal at $o$.
\begin{description}
\item[(i)] For $\alpha=\arcsin \frac{\varepsilon}{2R}$, if $\xi\in{\rm int}\,K$ with $\|\xi\|<\varepsilon/2$ and $u\in \Omega(v,\alpha)$, then $h_K(u)-\langle\xi,u\rangle<\varepsilon$.
\item[(ii)] If $\delta\in(0,\sin\alpha)$ and $\xi\in{\rm int}\,K$ satisfies $\|\xi\|\leq R\delta$, then 
$$
\int_{\Omega(v,\alpha)} (h_K(u)-\langle\xi,u\rangle)^{-q}\,d{\mathcal H}^{n-1}(u)
\geq \frac{(n-2)\kappa_{n-2}}{2^qR^q}\,\log\frac{\sin\alpha}{\delta}.
$$
\end{description}
\end{lemma}
\proof We may assume that 
$K=\{x\in RB^n:\,\langle x,v\rangle\leq 0\}$, and hence $h_K(u)=R\|u|v^\bot\|=R\sin\angle(u,v)$ if 
$u\in \Omega(v,\frac{\pi}2)$. In particular, $\alpha=\arcsin \frac{\varepsilon}{2R}$ works in (i).

For (ii), if $\delta\in(0,\sin\alpha)$,  $u\in \Omega(v,\alpha)$ with $\|u|v^\bot\|>\delta$, and 
$\|\xi\|<R\delta$,  then 
$h_K(u)-\langle\xi,u\rangle<2R||u|v^\bot\|$. We deduce that
 if $\|\xi\|<R\delta$ for $\xi\in{\rm int}\,K$,  then
\begin{eqnarray*}
\int_{\Omega(v,\alpha)} (h_K(u)-\langle\xi,u\rangle)^{-q}\,d{\mathcal H}^{n-1}(u)&\geq&
\int_{[(\sin\alpha\cdot B^n)\backslash (\delta B^n)]\cap v^\bot}
\frac{1}{2^qR^q\|x\|^{q}}\,d{\mathcal H}^{n-1}(x)\\
&=&\frac{(n-2)\kappa_{n-2}}{2^qR^q}
\int_\delta^{\sin\alpha} t^{n-2-q}\,dt\geq \frac{(n-2)\kappa_{n-2}}{2^qR^q}\,\log\frac{\sin\alpha}{\delta},
\end{eqnarray*}
which in turn yields the lemma. Q.E.D.\\

Let $K$ be  a convex body in $\R^n$. A point $p$ in its boundary is said to be \emph{smooth} if there exists a unique hyperplane supporting $K$ at $p$, and  $p$ is said to be \emph{singular} if it is not smooth.  We write $\partial' K$ and $\Xi_K$ to denote the set of smooth and singular points of $\partial K$, respectively. 
It is well known that $\mathcal{H}^{n-1}(\Xi_K)=0$. We call  $K$ \emph{quasi-smooth} if $\mathcal{H}^{n-1}({\mathbb S}^{n-1}\backslash\nu_K(\partial' K))=0$; 
namely, the set of  $u\in {\mathbb S}^{n-1}$ that are exterior normals only at singular points has $\mathcal{H}^{n-1}$-measure zero.

The following Lemma~\ref{Wulffvariation} will be used to prove first that the extremal convex body $K^\varepsilon$ is quasi-smooth
in Section~\ref{secquasi-smooth}, and secondly that it satisfies an Euler-Lagrange type equation
in Section~\ref{secEuler-Lagrange}. 
Let $K$ and $C$ be convex bodies containing the origin in their interior such that $rC\subset K$ for some $r>0$.
For $t\in(-r,r)$, we consider the \emph{Wulff shape}
$$
K_t=\{x\in\R^n:\,\langle x,u\rangle\leq h_K(u)+t h_C(u) \mbox{ \ for $u\in {\mathbb S}^{n-1}$}\},
$$
and we denote by $h_t$ the support function of $K_t$.

\begin{lemma}
\label{Wulffvariation}
Using the notation above, let $u\in {\mathbb S}^{n-1}$.
\begin{description}
\item[(i)] If $K\subset R\,B^n$ for $R>0$ and $t\in(-r,r)$, then
$|h_t(u)-h_K(u)|\leq \frac{R}r |t|$.
\item[(ii)] If $u$ is the exterior normal at some smooth point $z\in\partial K$, then
$$
\lim_{t\to 0}
\frac{h_t(u)-h_K(u)}{t}=h_C(u).
$$
\end{description}
\end{lemma}
\proof If $t\geq 0$ then $h_t=h_K+th_C$, therefore we may assume that $t<0$.

For (i), we observe that
$$
\left(1+\frac{t}r\right)K+|t|C\subset \left(1+\frac{t}r\right)K+\frac{|t|}r\cdot K=K.
$$
In other words, $\widetilde{K}_t=(1+\frac{t}r)K\subset K_t$, which in turn yields that
if $u\in {\mathbb S}^{n-1}$, then
$$
h_K(u)-h_t(u)\leq h_K(u)-h_{\widetilde{K}_t}(u)=\frac{|t|}r\cdot h_K(u)\leq \frac{R}r\cdot|t|.
$$

We turn to (ii). For $u\in {\mathbb S}^{n-1}$, we have $h_K(u)-h_t(u)\geq |t|\,h_C(u)$, and hence 
it is sufficient to prove that if $\varepsilon>0$ then
\begin{equation}
\label{iiWulffvariation}
h_K(u)-h_t(u)\leq (h_C(u)+\varepsilon)|t|
\end{equation}
provided that $t<0$ has small absolute value.
Let $D$ be the diameter of $C$, and let $\delta=\frac{\varepsilon}{\sqrt{D^2+\varepsilon^2}}$.
If $u$ is an exterior normal to $C$ at a point $q\in \partial C$, then $w=q+\varepsilon u$ satisfies
\begin{eqnarray}
\label{KC1}
\langle u,w\rangle &=&h_C(u)+\varepsilon\\
\label{KC2}
\langle u,x-w\rangle&\leq &-\delta\|x-w\| \mbox{ \ for all \ }x\in C.
\end{eqnarray}
Since $z\in\partial K$ is a smooth point with exterior unit normal $u$, there exists $\varrho>0$ such that if $\|x-z\|\leq\varrho$ and 
$\langle u,x-z\rangle\leq -\delta\|x-z\|$, then $x\in K$. We deduce from (\ref{KC2}) that if $(D+\varepsilon)|t|<\varrho$, then $y+|t|C\subset K$ for $y=z-|t|w$, and hence $y\in K_t$. Therefore
$$
h_K(u)-h_t(u)\leq \langle u,z-y\rangle=
 (h_C(u)+\varepsilon)|t|,
$$
proving (\ref{iiWulffvariation}). \ \ Q.E.D. \\

\bigskip

\noindent{\bf Remark.} Results similar to those proved in the previous lemma are contained in 
\cite[Section 3]{Jerison}.

\bigskip

Using the notation of Lemma~\ref{Wulffvariation}, if $K$ is quasi-smooth, then
$$
\lim_{t\to 0}
\frac{h_t(u)-h_K(u)}{t}=h_C(u)
$$
holds for $\mathcal{H}^{n-1}$ almost all $u\in {\mathbb S}^{n-1}$.
In particular, Lemma~\ref{xider} below applies.

\section{The energy function and optimal center}
\label{secenergy}

Let $p\in(-n,1)$. For $t>0$, we set
$$
\varphi(t)=\left\{ 
\begin{array}{ll}
t^p &\mbox{ \ if $p\in(0,1)$,}\\[0.5ex]
\log t&\mbox{ \ if $p=0$,}\\[0.5ex]
-t^p &\mbox{ \ if $p\in(-n,0)$.}
\end{array}
\right.
$$
The reasons behind this choice of $\varphi$ are that if $t\in(0,\infty)$, then
\begin{equation}
\label{phider}
\varphi'(t)=
\left\{ 
\begin{array}{ll}
|p|t^{p-1} &\mbox{ \ if $p\in(-n,1)\backslash \{0\}$}\\[0.5ex]
t^{p-1}&\mbox{ \ if $p=0$}
\end{array}
\right.
\end{equation}
is positive and decreasing, $\varphi$ is strictly increasing and $\varphi''$ is negative and continuous, and hence $\varphi$ is strictly concave.
In addition,
\begin{equation}
\label{philimit}
\lim_{t\to\infty}\varphi(t)=
\left\{ 
\begin{array}{ll}
\infty&\mbox{ \ if $p\in[0,1)$,}\\
0&\mbox{ \ if $p\in(-n,0)$.}
\end{array}
\right.
\end{equation}

Let $q=\max\{|p|,n-1\}$. In order to force the ``optimal center" of a convex body $K$ into its interior, we change $\varphi(t)$ into a function of order $-t^{-q}$ if $t$ is small (see Proposition~\ref{xiinside}).
For $t\in(0,1)$, the equation $\psi(s)=-t^{-(n-1)}+(n-1)t^{-n}(s-t)$ of the tangent to the graph of
$t\mapsto -t^{-(n-1)}$ satisfies $\psi(3t)\geq t^{-(n-1)}\geq 1$. Thus for any $\varepsilon\in(0,\frac13)$, there exists an increasing strictly 
concave function $\varphi_\varepsilon:\,(0,\infty)\to\R$, with continuous and negative second derivative, such that
\begin{equation}
\label{phiepsdef}
\varphi_\varepsilon(t)=\left\{ 
\begin{array}{ll}
\varphi(t)&\mbox{ \ if $t\geq 3\varepsilon$,}\\[1ex]
-t^{-q}&\mbox{ \ if $0<t\leq \varepsilon$,}
\end{array}
\right.
\end{equation}
and in addition
\begin{equation}
\label{t01p01}
\varphi_\varepsilon(t)\geq -t^{-q} \mbox{ \ if $t\in(0,1)$.}
\end{equation}
Let us observe that  if $p\in(-n,-(n-1)]$, we may choose $\varphi_\varepsilon=\varphi$.

Let $f$ be a measurable function on ${\mathbb S}^{n-1}$ such that there exist $\tau_2>\tau_1>0$ satisfying
\begin{equation}
\label{fcondition}
\tau_1<f(u)<\tau_2\mbox{ \ for $u\in {\mathbb S}^{n-1}$},
\end{equation}
and let $\mu$ be the Borel measure defined by $d\mu=f\,d\mathcal{H}^{n-1}$. 
\emph{We remark that, even when not explicitly stated, in all the results contained in Sections~3, 4, 5, 6 and 7 it is always assumed that \eqref{fcondition} holds.}

For $\varepsilon\in(0,\frac13)$, a convex body $K$ and $\xi\in{\rm int}\,K$, we define 
$$
\Phi_\varepsilon(K,\xi)=\int_{{\mathbb S}^{n-1}}\varphi_\varepsilon(h_K(u)-\langle u,\xi\rangle)\,d\mu(u).
$$

The proofs of Proposition~\ref{xiinside} and Lemma~\ref{xiKcontinuous} depend on the concavity of $\varphi_\varepsilon$ 
and the following Lemma~\ref{insidegood}. Here and throughout the paper, the convergence of sequence of convex bodies
is always meant in the sense of the Hausdorff metric.

\begin{lemma}
\label{insidegood}
Let $\{K_m\}$ be a sequence of convex bodies tending to a convex body $K$ in $\R^n$, and let
$\xi_m\in{\rm int}\,K_m$ be such that $\lim_{m\to \infty}\xi_m=z_0\in \partial K$. Then
$$
\lim_{m\to \infty}\Phi_\varepsilon(K_m,\xi_m)=-\infty.
$$
\end{lemma}
\proof Let $r_m>0$ be maximal such that
 $\xi_m+r_mB^n\subset K_m$, and let $y_m\in(\xi_m+r_mB^n)\cap\partial K_m$. The condition $z_0\in\partial K$
implies that $r_m=\|y_m-\xi_m\|$ tends to zero. Let $v_m\in {\mathbb S}^{n-1}$ be an exterior normal at $y_m$ to $K_m$.
For $R=1+{\rm diam}\,K$, we have ${\rm diam}K_m\leq R$ for large $m$; let 
$\alpha=\arcsin \frac{\varepsilon}{2R}$ be the constant of Lemma~\ref{oboundary}.
It follows
from Lemma~\ref{oboundary} (i) that if $u\in\Omega(v_m,\alpha)$ (the geodesic
ball on $\sfe$, centered at $v_m$ with opening $\alpha$), then
$h_{K_m}(u)-\langle u,\xi_m\rangle<\varepsilon$ for all $m$, and hence
$$
\varphi_\varepsilon(h_{K_m}(u)-\langle u,\xi_m\rangle)=-(h_{K_m}(u)-\langle u,\xi_m\rangle)^{-q}.
$$
Therefore Lemma~\ref{oboundary} (ii) and (\ref{fcondition}) yield that
\begin{equation}
\label{xiKcontinuous1}
\lim_{m\to\infty}\int_{\Omega(v_m,\alpha)}\varphi_\varepsilon(h_{K_m}(u)-\langle u,\xi_m\rangle)\,d\mu(u)
=-\infty.
\end{equation}
On the other hand,  
$\varphi_\varepsilon(h_{K_m}(u)-\langle u,\xi_m\rangle)\leq \varphi_\varepsilon(R)$
holds for all $m$ and $u\in {\mathbb S}^{n-1}$.
We deduce from (\ref{fcondition}) that
\begin{equation}
\label{xiKcontinuous2}
\int_{{\mathbb S}^{n-1}\backslash \Omega(v,\alpha)}\varphi_\varepsilon(h_{K_m}(u)-\langle u,\xi_m\rangle)\,d\mu(u)
<\tau_2n\kappa_n\varphi_\varepsilon(R)
\end{equation}
for all $m$. Combining (\ref{xiKcontinuous1}) and (\ref{xiKcontinuous2}) we conclude the proof. Q.E.D.\\

Now we single out the optimal $\xi\in{\rm int}\,K$.

\begin{prop}
\label{xiinside}
For $\varepsilon\in(0,\frac13)$ and a convex body $K$ in $\R^n$, there exists a unique $\xi(K)\in{\rm int}\,K$ such that
$$
\Phi_\varepsilon(K,\xi(K))=\max_{\xi\in{\rm int}\,K}\Phi_\varepsilon(K,\xi).
$$
\end{prop}
\proof Let $\xi_1,\xi_2\in{\rm int}\,K$, $\xi_1\neq\xi_2$, and let $\lambda\in(0,1)$. 
If $u\in {\mathbb S}^{n-1}\backslash(\xi_1-\xi_2)^\bot$, then $\langle u,\xi_1\rangle\neq\langle u,\xi_2\rangle$, and hence
the  strict concavity of $\varphi_\varepsilon$ yields that
$$
\varphi_\varepsilon(h_K(u)-\langle u,\lambda \xi_1+(1-\lambda)\xi_2\rangle)>
\lambda\varphi_\varepsilon(h_K(u)-\langle u,\xi_1\rangle)+
(1-\lambda)\varphi_\varepsilon(h_K(u)-\langle u,\xi_2\rangle).
$$
We deduce from (\ref{fcondition}) that
$$
\Phi_\varepsilon(K,\lambda \xi_1+(1-\lambda)\xi_2)>
\lambda\Phi_\varepsilon(K, \xi_1)+(1-\lambda)\Phi_\varepsilon(K,\xi_2),
$$
thus $\Phi_\varepsilon(K,\xi)$ is a strictly concave function of $\xi\in{\rm int}\,K$.

Let $\xi_m\in{\rm int}\,K$ such that 
$$
\lim_{m\to\infty}\Phi_\varepsilon(K,\xi_m)=\sup_{\xi\in{\rm int}\,K}\Phi_\varepsilon(K,\xi).
$$
We may assume that $\lim_{m\to\infty}\xi_m=z_0\in K$, and Lemma~\ref{insidegood} yields
 $z_0\in{\rm int}\,K$.
Since $\Phi_\varepsilon(K,\xi)$ is a strictly concave function of $\xi\in{\rm int}\,K$, we conclude
Proposition~\ref{xiinside}.
Q.E.D.\\

Since $\xi\mapsto \Phi_\varepsilon(K,\xi)$ is maximal at $\xi(K)\in{\rm int}\,K$, we deduce

\begin{corollary}
\label{intcond}
For $\varepsilon\in(0,\frac13)$ and a convex body $K$ in $\R^n$, we have
$$
\int_{{\mathbb S}^{n-1}}u\ \varphi'_\varepsilon\Big(h_{K}(u)-\langle u,\xi(K)\rangle\Big)\,d\mu(u)=o.
$$
\end{corollary}

An essential property of $\xi(K)$ is its continuity with respect to $K$.

\begin{lemma}
\label{xiKcontinuous}
For $\varepsilon\in(0,\frac13)$, both $\xi(K)$ and $\Phi_\varepsilon(K,\xi(K))$ are continuous functions of the convex body $K$ in $\R^n$.
\end{lemma}
\proof Let $\{K_m\}$ be a sequence convex bodies tending to a convex body $K$ in $\R^n$. We may assume that
$\lim_{m\to \infty}\xi(K_m)=z_0\in K$. 
There exists $r>0$ such that $\xi(K)+2r\,B^n\subset K$, and hence 
we may also assume that $\xi(K)+r\,B^n\subset K_m$ for all $m$. Thus
$$
\Phi_\varepsilon(K_m,\xi(K_m))\geq \Phi_\varepsilon(K_m,\xi(K))\geq \Phi_\varepsilon(\xi(K)+r\,B^n,\xi(K)),
$$
and in turn Lemma~\ref{insidegood} yields that
 $z_0\in{\rm int}\,K$.
It follows that $\varphi_\varepsilon(h_{K_m}(u)-\langle u,\xi(K_m)\rangle)$ tends uniformly to
$\varphi_\varepsilon(h_K(u)-\langle u,z_0\rangle)$. In particular,
$$
\Phi_\varepsilon(K,z_0)=\lim_{m\to\infty}\Phi_\varepsilon(K_m,\xi(K_m))\geq
\limsup_{m\to\infty}\Phi_\varepsilon(K_m,\xi(K))=\Phi_\varepsilon(K,\xi(K)).
$$
Since $\xi(K)$ is the unique maximum point of $\xi\mapsto \Phi_\varepsilon(K,\xi)$ on ${\rm int}\,K$
according to Proposition~\ref{xiinside}, we have $z_0=\xi(K)$. In turn, we conclude Lemma~\ref{xiKcontinuous}.
Q.E.D. \\

The next lemma shows that if we perturb a convex body $K$ in a differentiable way, then
$\xi(K)$ changes also in a differentiable way.

\begin{lemma} 
\label{xider}
For $\varepsilon\in(0,\frac13)$, let $c>0$ and $t_0>0$, and let $K_t$ be a family of convex bodies  with support function $h_t$ for $t\in[0,t_0)$. 
Assume that
\begin{enumerate}
 \item $|h_t(u)-h_0(u)|\leq c t$ for each $u\in {\mathbb S}^{n-1}$ and $t\in[0,t_0)$,
 \item $\lim_{t\to 0^+}\frac{h_t(u)-h_0(u)}{t}$ exists for $\mathcal{H}^{n-1}$-almost all $u\in {\mathbb S}^{n-1}$.
\end{enumerate}
Then $\lim_{t\to 0^+}\frac{\xi(K_t)-\xi(K_0)}{t}$ exists.
 \end{lemma}
\proof We may assume that $\xi(K_0)=o$. Since $\xi(K)\in{\rm int}\,K$ is the unique maximizer of 
$\xi\mapsto \Phi_\varepsilon(K,\xi)$, we deduce that
$$
\lim_{t\to 0^+}\xi(K_t)=o.
$$

Let $g(t,u)=h_t(u)-h_0(u)$ for $u\in {\mathbb S}^{n-1}$ and $t\in[0,t_0)$. In particular, there exists
constant $\gamma>0$ such that if $u\in {\mathbb S}^{n-1}$ and $t\in[0,t_0)$, then
$$
\varphi'_\varepsilon(h_t(u)-\langle u,\xi(K_t)\rangle)=
\varphi'_\varepsilon(h_0(u))+\varphi''_\varepsilon(h_0(u))\ \big(g(t,u)-\langle u,\xi(K_t)\rangle\big)
+e(t,u)
$$
where, setting $\gamma_1=2\gamma c^2$ and $\gamma_2=2\gamma$, we have
$$
|e(t,u)|\leq \gamma(g(t,u)-\langle u,\xi(K_t)\rangle)^2\leq \gamma(ct+\|\xi(K_t)\|)^2
\leq \gamma_1t^2+\gamma_2\|\xi(K_t)\|^2.
$$
In particular, $e(t,u)=e_1(t,u)+e_2(t,u)$ where 
\begin{equation}
\label{e1e2}
|e_1(t,u)|\leq \gamma_1t^2\mbox{ \ and \ }|e_2(t,u)|\leq \gamma_2\|\xi(K_t)\|^2.
\end{equation}
It follows from applying Corollary~\ref{intcond} to $K_t$ and $K_0$ that
$$
\int_{{\mathbb S}^{n-1}}u\  \Big(\varphi''_\varepsilon(h_0(u))\ \big(g(t,u)-\langle u,\xi(K_t)\rangle\big)
+e(t,u)\Big)\,d\mu(u)=o,
$$
which can be written as
$$
\int_{{\mathbb S}^{n-1}}u\  \big(\varphi''_\varepsilon(h_0(u))\  g(t,u)
+e_1(t,u)\big)\,d\mu(u)=
\int_{{\mathbb S}^{n-1}}u\  \langle u,\xi(K_t)\rangle\varphi''_\varepsilon(h_0(u))\,d\mu(u)
-\int_{{\mathbb S}^{n-1}}u\  e_2(t,u)\,d\mu(u).
$$
Since $\varphi''_\varepsilon(s)<0$ for all $s>0$, the symmetric matrix
$$
A=\int_{{\mathbb S}^{n-1}}u\otimes u\ \varphi''_\varepsilon(h_0(u))\,d\mu(u)
$$
is negative definite because for any $v\in {\mathbb S}^{n-1}$, we have
$$
v^TAv=\int_{{\mathbb S}^{n-1}}\langle u,v\rangle^2\ \varphi''_\varepsilon(h_0(u))\ f(u)\,d\mathcal{H}^{n-1}(u)<0.
$$
In addition, $A$ satisfies that
$$
\int_{{\mathbb S}^{n-1}}u\  \langle u,\xi(K_t)\rangle\ \varphi''_\varepsilon(h_0(u))\,d\mu(u)=A \,\xi(K_t).
$$
It follows from (\ref{e1e2}) that if $t$ is small, then
\begin{equation}
\label{psi1psi2}
A^{-1}\int_{{\mathbb S}^{n-1}}u\  \varphi''_\varepsilon(h_0(u))\  g(t,u)\,d\mu(u)
+\psi_1(t)=\xi(K_t)-\psi_2(t),
\end{equation}
where $\|\psi_1(t)\|\leq \alpha_1 t^2$ and $\|\psi_2(t)\|\leq \alpha_2 \|\xi(K_t)\|^2$
for constants $\alpha_1,\alpha_2>0$. Since $\xi(K_t)$ tends to $o$, 
if $t$ is small, then $\|\xi(K_t)-\psi_2(t)\|\geq \frac12\,\|\xi(K_t)\|$, thus
$\|\xi(K_t)\|\leq \beta\,t$ for a constant $\beta>0$ by $g(t,u)\leq ct$.
In particular, $\|\psi_2(t)\|\leq \alpha_2 \beta^2t^2$. 
 Since
there exists $\lim_{t\to 0^+}\frac{g(t,u)-g(0,u)}{t}=\partial_1 g(0,u)$
for $\mu$ almost all 
$u\in {\mathbb S}^{n-1}$, and $\frac{g(t,u)-g(0,u)}{t}<c$ for all $u\in {\mathbb S}^{n-1}$ and $t>0$, we conclude that
$$
\left.\frac{d}{dt}\,\xi(K_t)\right|_{t=0}=
A^{-1}\int_{{\mathbb S}^{n-1}}u\  \varphi''_\varepsilon(h_0(u))\  \partial_1g(0,u)\,d\mu(u).
$$
Q.E.D.

\begin{corollary}
\label{center-irrelevant}
Under the conditions of Lemma~\ref{xider}, and denoting $K_0$ by $K$, we have
$$
\left.\frac{d}{dt}\,\Phi_\varepsilon(K_t,\xi(K_t))\right|_{t=0}=\int_{{\mathbb S}^{n-1}}\left.\frac{\partial}{\partial t} h_{K_t}(u)\right|_{t=0} \varphi'_\varepsilon\big(h_K(u)-\langle u,\xi(K)\rangle\big)\,d\mu(u).
$$
\end{corollary}
\proof We write $h(t,u)=h_{K_t}(u)$ and $\xi(t)=\xi(K_t)$; Corollary~\ref{intcond} and Lemma~\ref{xider} yield
\begin{eqnarray*}
\left.\frac{d}{dt}\,\Phi_\varepsilon(K_t,\xi(K_t))\right|_{t=0}&=&
\left.\frac{d}{dt}\, \int_{{\mathbb S}^{n-1}}\varphi_\varepsilon\big(h(t,u)-\langle u,\xi(t)\rangle\big)\,d\mu(u)\right|_{t=0}\\
&=&\int_{{\mathbb S}^{n-1}}\partial_1h(0,u)\ \varphi'_\varepsilon\big(h_K(u)-\langle u,\xi(K)\rangle\big)\,d\mu(u)-\\
&&\int_{{\mathbb S}^{n-1}}\langle u,\xi'(0)\rangle\ \varphi'_\varepsilon\big(h_K(u)-\langle u,\xi(K)\rangle\big)\,d\mu(u)\\
&=&\int_{{\mathbb S}^{n-1}}\partial_1h(0,u)\ \varphi'_\varepsilon\big(h_K(u)-\langle u,\xi(K)\rangle\big)\,d\mu(u).
\end{eqnarray*}
Q.E.D.

\section{The existence of the minimum convex body $K^\varepsilon$}
\label{secextremal-existence}

Let $p\in(-n,1)$, and let $\mathcal{K}_1\subset \mathcal{K}_0^n$ be the set of convex bodies with volume one and containing the origin.

We observe that $\kappa_n^{-1/n}>\frac12$, $\kappa_n^{-1/n}B^n\in \mathcal{K}_1$ and the diameter of $\kappa_n^{-1/n}B^n$ is
$2\kappa_n^{-1/n}$. It follows from $\varphi_\varepsilon\leq \varphi$ and the monotonicity of $\varphi$, that if $\varepsilon\in(0,\frac16)$, then
\begin{eqnarray}
\label{Phiball}
\Phi_\varepsilon(\kappa_n^{-1/n}B^n,\xi(\kappa_n^{-1/n}B^n))&\leq& \int_{{\mathbb S}^{n-1}}\varphi(2\kappa_n^{-1/n})d\mu
=\varphi(2\kappa_n^{-1/n})\mu({\mathbb S}^{n-1})\\
\nonumber
&\leq&
\left\{\begin{array}{ll}
 2^{p}\kappa_n^{\frac{-p}n}\,n\kappa_n\cdot \tau_2&\mbox{ if  
$p\in(0,1)$,}\\[1ex]
 \log\big(2\kappa_n^{\frac{-1}n}\big)\,n\kappa_n\cdot \tau_2 &\mbox{ if $p=0$,}\\[1ex]
 -2^{p}\kappa_n^{\frac{-p}n}\,n\kappa_n \cdot \tau_1&\mbox{ if  
$p\in(-n,0)$.}
\end{array}
\right.
\end{eqnarray}

\bigskip

For $K\in \mathcal{K}_1$,  let
$R(K)=\max\{\|x-\sigma(K)\|:\,x\in K\}$. We define the measure of the empty set to be zero.
We note that if $\alpha\in (0,\frac{\pi}2)$ and $v\in {\mathbb S}^{n-1}$, then
\begin{equation}
\label{caparea}
\mathcal{H}^{n-1}\left(\{u\in {\mathbb S}^{n-1}:\,\langle u,v\rangle\geq \cos\alpha\} \right)
\geq (\sin\alpha)^{n-1}\kappa_{n-1}.
\end{equation}

\begin{lemma}
\label{Knotbounded01}
Let $p\in[0,1)$. There exists $R_0>1$, depending on $n$, $p$, $\tau_1$ and $\tau_2$, such that if
$K\in \mathcal{K}_1$, $R(K)>R_0$ and $\varepsilon\in(0,\frac16)$, then
$$
\Phi_\varepsilon(K,\xi(K))>\Phi_\varepsilon(\kappa_n^{-1/n}B^n,\xi(\kappa_n^{-1/n}B^n)).
$$
\end{lemma}
\proof Let $K\in\mathcal{K}_1$. We may assume $\sigma(K)=o$ and $R=R(K)>2n$. 
Let $v\in {\mathbb S}^{n-1}$ satisfy $Rv\in K$. It follows from Lemma~\ref{centroid} (i) that $(-R/n)v\in K$, as well. 

We write $c_0,c_1$ to denote positive constants depending on $n,p,\tau_1,\tau_2$.
We consider
$$
\Xi_0=\{u\in {\mathbb S}^{n-1}:\, h_K(u)< 1\},
$$
and $\Xi_1={\mathbb S}^{n-1}\backslash \Xi_0$. We observe that if $u\in \Omega(v,\frac{\pi}3)$, then 
$h_K(u)\geq \langle u,Rv\rangle\geq R/2$, and in turn $ \Omega(v,\frac{\pi}3)\subset \Xi_1$.  Since 
$\mu(\Omega(v,\frac{\pi}3))\geq \tau_1(\frac{\sqrt{3}}2)^{n-1}\kappa_{n-1}$ by (\ref{caparea}) and $\varphi_\varepsilon(t)=\varphi(t)>0$ for $t>1$, we have
\begin{equation}
\label{intOmega}
\int_\Xi \varphi_\varepsilon\circ h_K\,d\mu\geq 
\int_{\Omega(v,\frac{\pi}3)} \varphi_\varepsilon\circ h_K\,d\mu\geq 
\tau_1\left(\frac{\sqrt{3}}2\right)^{n-1}\kappa_{n-1}\varphi(R/2)=c_1\varphi(R/2).
\end{equation}
However, if $u\in\Xi_0$, then $|\langle u,v\rangle|<n/R$ as $1>h_K(u)\geq |\langle (R/n)v,u\rangle|$. It follows that
\begin{equation}
\label{Omega0}
\mathcal{H}^{n-1}(\Xi_0)\leq (n-1)\kappa_{n-1}\cdot \frac{2n}R<(n-1)\kappa_{n-1}.
\end{equation}
We deduce from (\ref{t01p01}), the H\"older inequality,
the Blaschke-Santal\'o inequality Lemma~\ref{centroid} (ii) and (\ref{Omega0}) that 
\begin{eqnarray}
\nonumber
\int_{\Xi_0} \varphi_\varepsilon\circ h_K\,d\mu&\geq &
 -\tau_2\int_{\Xi_0} h_K^{-(n-1)}\,d\mathcal{H}^{n-1}\\
\nonumber
&\geq&-\tau_2\left(\int_{\Xi_0} h_K^{-n}\,d\mathcal{H}^{n-1}\right)^{\frac{n-1}n}
\mathcal{H}^{n-1}(\Xi_0)^{\frac1n}\\
\label{intOmega0}
&\geq& -\tau_2(n\kappa_n^2)^{\frac{n-1}n}((n-1)\kappa_{n-1})^{\frac1n}=-c_0.
\end{eqnarray}
Writing $c(n,p,\tau_1,\tau_2)$ to denote the constant on the right hand side of (\ref{Phiball}),
comparing (\ref{Phiball}), (\ref{intOmega}) and (\ref{intOmega0}) yields 
$$
c_1\varphi(R/2)
-c_0\leq c(n,p,\tau_1,\tau_2),
$$
and, in turn, the existence of $R_0$ as $\lim_{R\to\infty}\varphi(R/2)=\infty$ by (\ref{philimit}). \ \ Q.E.D.\\

The argument in the case $p\in(-n,0)$ is similar to the previous one, but it needs to be refined as now
$\lim_{t\to\infty}\varphi(t)=0$.

\begin{lemma}
\label{Knotboundedn1}
Let $p\in(-n,0)$. There exists $R_0>1$, depending on $n$, $p$, $\tau_1$ and $\tau_2$, such that if $K\in \mathcal{K}_1$, 
$R(K)>R_0$, and $\varepsilon\in(0,\frac16)$, then
$$
\Phi_\varepsilon(K,\xi(K))>\Phi_\varepsilon(\kappa_n^{-1/n}B^n,\xi(\kappa_n^{-1/n}B^n)).
$$
\end{lemma}
\proof  Let $K\in\mathcal{K}_1$. We may assume $\sigma(K)=o$ and $R=R(K)>4n^2$. Let $v\in {\mathbb S}^{n-1}$
satisfy $Rv\in K$. It follows from Lemma~\ref{centroid} (i) that $(-R/n)v\in K$, as well. 

In this case, we divide ${\mathbb S}^{n-1}$ into three parts:
\begin{eqnarray*}
\Xi_0&=&\{u\in {\mathbb S}^{n-1}:\, h_K(u)< 1\},\\
\Xi_1&=&\{u\in {\mathbb S}^{n-1}:\, 1\leq h_K(u)<\sqrt{R}\},\\
\Xi_2&=&\{u\in {\mathbb S}^{n-1}:\, h_K(u)\geq \sqrt{R}\}.
\end{eqnarray*}

If $u\in\Xi_0\cup \Xi_1$, then
$$
\sqrt{R}>h_K(u)\geq \max\{\langle u,Rv\rangle, \langle u,(-R/n)v\rangle\}\geq
(R/n)|\langle u,v\rangle|.
$$
Thus $|\langle u,v\rangle|\leq n/\sqrt{R}$, which in turn yields that
\begin{equation}
\label{Xi01pn1}
\mathcal{H}^{n-1}(\Xi_0\cup \Xi_1)\leq \frac{4n(n-1)\kappa_{n-1}}{\sqrt{R}}.
\end{equation}
We write $c_0,c_1,c_2$ to denote positive constants depending on $n,p,\tau_1,\tau_2$.
If $u\in\Xi_0$, then $\varphi_\varepsilon(h_K(u))\geq -h_K(u)^{-q}$ according to (\ref{t01p01}),
and hence we deduce from  the H\"older inequality,
the Blaschke-Santal\'o inequality Lemma~\ref{centroid} (ii) and (\ref{Xi01pn1}) that
\begin{eqnarray}
\nonumber
\int_{\Xi_0} \varphi_\varepsilon\circ h_K\,d\mu&\geq &
 -\tau_2\int_{\Xi_0} h_K^{-q}\,d\mathcal{H}^{n-1}\\
\nonumber
&\geq&-\tau_2\left(\int_{\Xi_0} h_K^{-n}\,d\mathcal{H}^{n-1}\right)^{\frac{q}n}
\mathcal{H}^{n-1}(\Xi_0)^{\frac{n-q}n}\\
\label{intXi0}
&\geq& -\tau_2(n\kappa_n^2)^{\frac{q}n}\left(\frac{4n(n-1)\kappa_{n-1}}{\sqrt{R}}\right)^{\frac{n-q}n}
=-c_0R^{-\frac{n-q}{2n}}.
\end{eqnarray}
Next if $u\in\Xi_1$, then $\varphi_\varepsilon(h_K(u))= -h_K(u)^{-|p|}$,
and hence we deduce from  the H\"older inequality,
the Blaschke-Santal\'o inequality Lemma~\ref{centroid} (ii) and (\ref{Xi01pn1}) that
\begin{eqnarray}
\nonumber
\int_{\Xi_1} \varphi_\varepsilon\circ h_K\,d\mu&\geq &
 -\tau_2\int_{\Xi_1} h_K^{-|p|}\,d\mathcal{H}^{n-1}\\
\nonumber
&\geq&-\tau_2\left(\int_{\Xi_1} h_K^{-n}\,d\mathcal{H}^{n-1}\right)^{\frac{|p|}n}
\mathcal{H}^{n-1}(\Xi_1)^{\frac{n-|p|}n}\\
\label{intXi1}
&\geq& -\tau_2(n\kappa_n^2)^{\frac{|p|}n}\left(\frac{4n(n-1)\kappa_{n-1}}{\sqrt{R}}\right)^{\frac{n-|p|}n}
=-c_1R^{-\frac{n-|p|}{2n}}.
\end{eqnarray}
Finally, if $u\in\Xi_2$, then $\varphi_\varepsilon(h_K(u))\geq  \varphi_\varepsilon(\sqrt{R})$, and hence
\begin{equation}
\label{intXi2}
\int_{\Xi_2} \varphi_\varepsilon\circ h_K\,d\mu\geq \tau_2n\kappa_n\cdot \varphi_\varepsilon(\sqrt{R})
=c_2\varphi_\varepsilon(\sqrt{R}).
\end{equation}
Writing $c(n,p,\tau_1,\tau_2)<0$ to denote the constant on the right hand side of (\ref{Phiball}) in the case
$p\in(-n,0)$, comparing (\ref{Phiball}), (\ref{intXi0}), (\ref{intXi1}) and (\ref{intXi2}) yields 
$$
-c_0R^{-\frac{n-q}{2n}}-c_1R^{-\frac{n-|p|}{2n}}
+c_2\varphi_\varepsilon(\sqrt{R})\leq c(n,p,\tau_1,\tau_2)<0,
$$
and in turn the existence of $R_0$
as $\lim_{R\to\infty}\varphi(\sqrt{R})=0$ by (\ref{philimit}). \ \ Q.E.D.\\

We deduce from the Blaschke selection theorem and the continuity of $\Phi_\varepsilon(K,\xi(K))$ 
(see Lemma~\ref{xiKcontinuous}) the existence of the extremal body $K^\varepsilon$.

\begin{corollary}
\label{Kepsilon}
For every $\varepsilon\in(0,\frac16)$, if $R_0>0$ is the number depending on $n$, $p$, $\tau_1$ and $\tau_2$ of 
Lemma~\ref{Knotbounded01} and Lemma~\ref{Knotboundedn1}, there exists
 $K^\varepsilon\in \mathcal{K}_1$ with $R(K^\varepsilon)\leq R_0$, such that
$$
\Phi_\varepsilon(K^\varepsilon,\xi(K^\varepsilon))=\min_{K\in \mathcal{K}_1}\Phi_\varepsilon(K,\xi(K)).
$$
\end{corollary}

\section{$K^\varepsilon$ is quasi-smooth}
\label{secquasi-smooth}

 Lemma~\ref{Wulffvariation-inwards} below is essential in order to apply Lemma~\ref{xider}. For
any convex body $K$ and $\omega\subset {\mathbb S}^{n-1}$, we define
$$
\nu_K^{-1}(\omega)=\{x\in\partial K:\,\nu_K(x)\cap \omega\neq\emptyset\}.
$$
For $u\in {\mathbb S}^{n-1}$, we write $F(K,u)$ to denote the face of $K$ with exterior unit normal $u$; in other words,
$$
F(K,u)=\{x\in\partial K:\,\langle x,u\rangle=h_K(u)\}.
$$

\begin{lemma}
\label{Wulffvariation-inwards}
Let $K$  be a convex body with $rB^n\subset{\rm int}\, K$ for $r>0$, let $\omega\subset {\mathbb S}^{n-1}$ be closed, and let
$$
K_t=\{x\in K:\,\langle x,v\rangle\leq h_{K}(v)-t\mbox{ \ \ for every $v\in\omega$}\}
$$
for $t\in(0,r)$. If  $h_t$ is the support function of $K_t$,  then
$\lim_{t\to 0^+}
\frac{h_t(u)-h_K(u)}{t}$ exists for all $u\in {\mathbb S}^{n-1}$. 
\end{lemma}
\noindent{\bf Remark } Readily, $\lim_{t\to 0^+} \frac{h_t(u)-h_K(u)}{t}\leq -1$ if $u\in\omega$.
\proof We set $X=\nu_K^{-1}(\omega)$; this is a compact set. We consider two cases: 
either $u$ is an exterior unit normal at some $y\not\in X$, or $F(K,u)\subset X$.

In the first case
 $h_t(u)=h_K(u)$ for sufficiently small $t$,  and hence
$\lim_{t\to 0} \frac{h_t(u)-h_K(u)}{t}=0$.

Next let $F(K,u)\subset X$ for $u\in {\mathbb S}^{n-1}$, and let 
$z\in {\rm relint}\,F(K,u)$. We define $\Sigma$ to be the support cone at $z$; namely,
$$
\Sigma={\rm cl}\{ \alpha (y-z):\,y\in K\mbox{ \ and \ }\alpha\geq 0\}
=\{y\in \R^n:\,\langle y,v\rangle\leq 0\mbox{ \ for \ }v\in \nu_K(z)\}.
$$
For small $t>0$, let
$$
C_t=\{x\in \Sigma:\, \langle x,v\rangle\leq -t\mbox{ \ for \ }v\in\omega\cap\nu_K(z)\};
$$
note that $C_t$ is a closed convex set satisfying $K_t-z\subset C_t$, and $C_t=tC_1$. We define
$$
\aleph =\sup\{\langle x,u\rangle:\,x\in C_1\}\leq 0,
$$
and claim that for any $\tau>0$  there exists $t_0>0$ depending on $z$, $K$ and $\tau$ such that if $t\in(0,t_0)$, then
\begin{equation}
\label{inward-t-tau}
(\aleph-\tau) t\leq h_t(u)-h_K(u)\leq \aleph t.
\end{equation}
 To prove (\ref{inward-t-tau}), we may assume that $z=o$, and hence 
$h_K(v)=0$ for all $v\in \nu_K(z)$. 
For  the upper bound in (\ref{inward-t-tau}), we observe that $K_t\subset C_t$, and hence
$$
h_t(u)-h_K(u)=h_t(u)\leq \sup\{\langle x,u\rangle:\,x\in C_t\}=\aleph t.
$$

For the lower bound, let $y_\tau\in {\rm int}\, C_1$ be such that
$$
\langle y_\tau,u\rangle>\aleph-\tau.
$$
Since $\omega\cap\nu_K(o)$ is compact, there exists $\delta>0$ such that
$$
\langle y_\tau,v\rangle\leq -1-\delta \mbox{ \ for }v\in\omega\cap\nu_K(o).
$$
Moreover, $y_\tau\in{\rm int}\,\Sigma$ yields the existence of $t_1>0$ such that
$ty_\tau\in K$ if $t\in(0,t_1]$. 

We also need one more constant reflecting the boundary structure of $K$ near $o$.
Recall that $h_K(w)\geq 0$ for all $w\in {\mathbb S}^{n-1}$, and $h_K(w)=0$  if and only if $w\in \nu_K(o)$. Since $\omega$ is compact, there exists $\gamma>0$ such that 
$$
\mbox{if $w\in\omega$ and $\|w-v\|\geq \delta/\|y_\tau\|$ for all 
$v\in\omega\cap\nu_K(o)$, then $h_K(w)\geq \gamma$.}
$$
We finally define $t_0\in(0,t_1]$ by the condition $t_0 \|y_\tau\|+t_0<\gamma$.

Let $t\in(0,t_0)$, and hence $ty_\tau\in K$. If $w\in \omega$ satisfies 
 $\|w-v\|\geq \delta/\|y_\tau\|$ for all 
$v\in\omega\cap\nu_K(o)$, then
$$
\langle ty_\tau, w\rangle\leq t_0 \|y_\tau\|<\gamma-t_0<h_K(w)-t.
$$
However, if $w\in \omega$ and there exists $v\in\omega\cap\nu_K(o)$ satisfying 
$\|w-v\|< \delta/\|y_\tau\|$, then
$$
\langle ty_\tau, w\rangle=\langle ty_\tau, w-v\rangle+\langle ty_\tau, v\rangle\leq 
t\delta+t(-1-\delta)=-t\leq h_K(w)-t.
$$
We deduce that $ty_\tau\in K_t$, thus 
$$
h_t(u)-h_K(u)\geq \langle ty_\tau, u\rangle\geq (\aleph-\tau) t,
$$
concluding the proof of (\ref{inward-t-tau}).

In turn, (\ref{inward-t-tau}) yields that  $\lim_{t\to 0^+}
\frac{h_t(u)-h_K(u)}{t}=\aleph$. Q.E.D.\\

A crucial fact for us is Alexandrov's Lemma~\ref{Alexandrov} (see Lemma 7.5.3 in \cite{SCH}). To state this, let
$g: (-r,r)\times {\mathbb S}^{n-1} \to \R$, $r>0$, verify
\begin{itemize}
\item $g(0,u)=h_K(u)$ for a convex body $K$;
\item for every $u\in\sfe$ the limit $\lim_{t\to 0}\frac{g(t,u)-g(0,u)}{t}=\partial_1g(0,u)$ exists (finite) and the convergence is uniform 
with respect to $u\in {\mathbb S}^{n-1}$; moreover
$\partial_1g(0,u)$ is continuous with respect to $u\in {\mathbb S}^{n-1}$;
\item $K_t=\{x\in \R^n:\,\langle x,u\rangle \leq g(t,u)\mbox{ \ for any }u\in {\mathbb S}^{n-1}\}$ is a convex body
for $t\in(-r,r)$.
\end{itemize}

\begin{lemma}[Alexandrov]
\label{Alexandrov}
In the notation introduced above, we have
$$
\lim_{t\to 0}\frac{V(K_t)-V(K)}t=\int_{{\mathbb S}^{n-1}}\partial_1g(0,u)\,d S_K(u).
$$
\end{lemma}

Next we present a way to improve on $\Phi_\varepsilon(K,\xi(K))$ while staying in the family $\mathcal{K}_1$.

\begin{prop}
\label{improve}
If for $K\in\mathcal{K}_1$ there exists
a closed set $\omega\subset {\mathbb S}^{n-1}$ with $\mathcal{H}^{n-1}(\omega)>0$, such that $S_K(\omega)=0$, then there exists
a  convex body $\widetilde{K}\in\mathcal{K}_1$ such that
$\Phi_\varepsilon(\widetilde{K},\xi(\widetilde{K}))<\Phi_\varepsilon(K,\xi(K))$.
\end{prop}
\proof
 For small
$t\geq 0$, we consider
$$
K_t=\{x\in K:\,\langle x,u\rangle\leq h_{K}(u)-t\mbox{ \ \ for $u\in\omega$}\},
$$
and
$$
\widetilde{K}_t=V(K_t)^{-1/n}K_t\in\mathcal{K}_1.
$$
We define $\alpha(t)=V(K_t)^{-1/n}$, so that in particular $\alpha(0)=1$. We claim that
\begin{equation}
\label{alphaprime}
\alpha'(0)=0.
\end{equation}
Since $\alpha$ is monotone decreasing, it is equivalent to prove that if $\eta\in(0,1)$, then
\begin{equation}
\label{alphaprimeeta}
\liminf_{t\to 0^+}\frac{V(K_t)-V(K)}t\geq -\eta.
\end{equation}
Since $S_K(\omega)=0$ and $\omega$ is closed, we can choose a continuous function $\psi:\,{\mathbb S}^{n-1}\to[0,1]$ such that
$\psi(u)=1$ if $u\in\omega$, and
$$
\int_{{\mathbb S}^{n-1}}\psi\,dS_K\leq \eta.
$$
 For small
$t>0$, we consider $\gamma_t=h_K-t\psi$ and
$$
K_{\psi,t}=\{x\in K:\,\langle x,u\rangle\leq \gamma_t(u)\mbox{ \ \ for $u\in\omega$}\},
$$
and hence $K_{\psi,t}\subset K_t$. Using Lemma~\ref{Alexandrov}, we deduce that
$$
\liminf_{t\to 0^+}\frac{V(K_t)-V(K)}t\geq \left.\frac{d}{dt}V(K_{\psi,t})\right|_{t=0^+}=
-\int_{{\mathbb S}^{n-1}}\psi\,dS_K\geq -\eta.
$$
We conclude (\ref{alphaprimeeta}), and in turn (\ref{alphaprime}).

We set $h(t,u)=h_{K_t}(u)$. As
$$
K_{0,t}=\{x\in K:\,x+tB^n\subset K\}\subset K_t,
$$
Lemma~\ref{Wulffvariation} (i), with $C=B^n$, yields that there is $c>0$ such that if $t>0$ is small,
then 
$$
-ct\leq h_{K_{0,t}}(u)-h_K(u)\leq h(t,u)-h(0,u)\leq 0
$$ 
for any $u\in {\mathbb S}^{n-1}$. In addition, we deduce from
Lemma~\ref{Wulffvariation-inwards} that
$\lim_{t\to 0^+}\frac{h(t,u)-h(0,u)}t=\partial_1 h(0,u)\leq 0$
exists for any $u\in {\mathbb S}^{n-1}$ where $\partial_1 h(0,u)\leq -1$ for $u\in\omega$ by definition.
Next let $\tilde{h}(t,u)=\alpha(t)h(t,u)=h_{\widetilde{K}_t}(u)$ for $u\in {\mathbb S}^{n-1}$
and small $t>0$. Therefore there exists $\tilde{c}>0$ such that if $t>0$ is small,
then $|\tilde{h}(t,u)-\tilde{h}(0,u)|\leq \tilde{c}t$ for any $u\in {\mathbb S}^{n-1}$, and $\alpha(0)=1$ and
(\ref{alphaprime}) implies that
$$
\lim_{t\to 0^+}\frac{\tilde{h}(t,u)-\tilde{h}(0,u)}t=\partial_1 \tilde{h}(0,u)=\partial_1 h(0,u)\leq 0
$$
exists for any $u\in {\mathbb S}^{n-1}$, where $\partial_1 \tilde{h}(0,u)\leq -1$ for $u\in\omega$.
We may assume that $\xi(K)=o$ and $K\subset RB^n$ for $R>0$ where $K=\tilde{K}_0$.
As $\varphi'_\varepsilon$ is positive and monotone decreasing, $\mathcal{H}^{n-1}(\omega)>0$ and Corollary~\ref{center-irrelevant} imply
$$
\left.\frac{d}{dt}\Phi_\varepsilon(\widetilde{K}_t,\xi(\widetilde{K}_t))\right|_{t=0}=
\int_{{\mathbb S}^{n-1}}\partial_1\tilde{h}(0,u)\cdot \varphi'_\varepsilon(h_K(u))\,d\mu(u)\leq
\int_\omega (-1)\varphi'_\varepsilon(R)\,d\mu(u)<0.
$$
Therefore $\Phi_\varepsilon(\widetilde{K}_t,\xi(\widetilde{K}_t))<\Phi_\varepsilon(K,\xi(K))$ for small $t>0$, 
which proves Lemma~\ref{improve}. Q.E.D.\\

\begin{corollary}
\label{quasi-smooth}
$K^\varepsilon$ is quasi-smooth.
\end{corollary}
\proof Let $\partial'K$ and $\Xi_K$ be as in the definition of quasi-smooth body, immediately after the proof of Lemma~\ref{oboundary}. If $K\in \mathcal{K}_1$ is not quasi-smooth, then 
$\mathcal{H}^{n-1}({\mathbb S}^{n-1}\backslash \nu_K(\partial'K))>0$.
Now there exists a closed set $\omega\subset {\mathbb S}^{n-1}\backslash \nu_K(\partial'K)$ 
such that $\mathcal{H}^{n-1}(\omega)>0$. If an exterior normal at  
 $x\in\partial K$ lies in $\omega$, then $x\in\Xi_K$, and hence
 $S_K(\omega)\leq \mathcal{H}^{n-1}(\Xi_K)=0$.
Thus  
Proposition~\ref{improve} 
yields the existence of a  convex body $\widetilde{K}\in\mathcal{K}_1$ such that
$\Phi(\widetilde{K},\xi(\widetilde{K}))<\Phi(K,\xi(K))$.
We conclude that $K^\varepsilon$ is quasi-smooth by its extremality property. Q.E.D.

\section{The variational formula (to  get $\lambda_\varepsilon$)}
\label{secEuler-Lagrange}

We define
\begin{equation}
\label{lambdaepsilon}
\lambda_\varepsilon=\frac1n\int_{{\mathbb S}^{n-1}}h_{K^\varepsilon-\xi(K^\varepsilon)}(u)\cdot
\varphi'_\varepsilon(h_{K^\varepsilon-\xi(K^\varepsilon)}(u))\,d\mu(u).
\end{equation}

\begin{prop}
\label{Euler-Lagrange}
$\varphi'_\varepsilon(h_{K^\varepsilon}(u)-\langle\xi(K^\varepsilon),u\rangle)\,d\mu(u)=
\lambda_\varepsilon\,dS_{K^\varepsilon}$ as measures on ${\mathbb S}^{n-1}$.
\end{prop}
\proof To simplify the argument, we write $K=K^\varepsilon$, and assume that $\xi(K)=o$. First we claim that if $C$ is any convex body with $o\in{\rm int} C$, then
\begin{equation}
\label{support}
\int_{{\mathbb S}^{n-1}}h_C\lambda_\varepsilon\,dS_{K}=\int_{{\mathbb S}^{n-1}}h_C(u)\varphi'_\varepsilon(h_K(u))\,d\mu(u).
\end{equation}
Assuming $rC\subset K$ for $r>0$, if $t\in(-r,r)$, then 
we consider
$$
K_t=\{x\in K:\,\langle x,u\rangle\leq h_{K}(u)+th_C(u)\mbox{ \ \ for $u\in {\mathbb S}^{n-1}$}\},
$$
and
$$
\widetilde{K}_t=V(K_t)^{-1/n}K_t\in\mathcal{K}_1.
$$
We define $\alpha(t)=V(K_t)^{-1/n}$, so that in particular $\alpha(0)=1$.
Lemma~\ref{Alexandrov} yields that
$$
\left.\frac{d}{dt}V(K_t)\right|_{t=0}=\int_{{\mathbb S}^{n-1}}h_C\,dS_K,
$$
and hence
\begin{equation}
\label{alphader0}
\alpha'(0)=\frac{-1}n\int_{{\mathbb S}^{n-1}}h_C\,dS_K.
\end{equation}

We write $h(t,u)=h_{K_t}(u)$. Since $K$ is quasi-smooth,
Lemma~\ref{Wulffvariation} (i) and (ii)
imply that there exists $c>0$ such that if $t\in(-r,r)$,
then $|h(t,u)-h(0,u)|\leq c |t|$ for any $u\in {\mathbb S}^{n-1}$, and 
$\lim_{t\to 0}\frac{h(t,u)-h(0,u)}t=h_C(u)$
exists for $\mathcal{H}^{n-1}$-a.e. $u\in {\mathbb S}^{n-1}$. 
Next let $\tilde{h}(t,u)=\alpha(t)h(t,u)=h_{\widetilde{K}_t}(u)$ for $u\in {\mathbb S}^{n-1}$
and $t\in (-r,r)$. From the properties of $h(t,u)$ mentioned above and (\ref{alphader0})
it follows the existence of $\tilde{c}>0$ such that if $t\in (-r,r)$,
then $|\tilde{h}(t,u)-\tilde{h}(0,u)|\leq \tilde{c}|t|$ for any $u\in {\mathbb S}^{n-1}$, and 
$$
\lim_{t\to 0}\frac{\tilde{h}(t,u)-\tilde{h}(0,u)}t=\partial_1 \tilde{h}(0,u)=\alpha'(0)h_K(u)+h_C(u)
$$
for any $u\in {\mathbb S}^{n-1}$.
As $\Phi(\widetilde{K}_t,\xi(\widetilde{K}_t))$ has a minimum at $t=0$ by the extremal property of 
$K^\varepsilon=\widetilde{K}_0=K$,
Corollary~\ref{center-irrelevant} implies
\begin{eqnarray*}
0&=&\left.\frac{d}{dt}\Phi(\widetilde{K}_t,\xi(\widetilde{K}_t))\right|_{t=0}=
\int_{{\mathbb S}^{n-1}}\partial_1\tilde{h}(0,u)\cdot \varphi'_\varepsilon(h_K(u))\,d\mu(u)\\
&=&\int_{{\mathbb S}^{n-1}}(\alpha'(0)h_K(u)+h_C(u))\varphi'_\varepsilon(h_K(u))\,d\mu(u)\\
&=&\int_{{\mathbb S}^{n-1}}h_C(u)\varphi'_\varepsilon(h_K(u))\,d\mu(u)-\int_{{\mathbb S}^{n-1}}h_C\lambda_\varepsilon\,dS_K,
\end{eqnarray*}
and in turn we deduce (\ref{support}).

Since differences of support functions are dense among continuous functions on ${\mathbb S}^{n-1}$ (see {e.g.} \cite{SCH}), 
we have
$$
\int_{{\mathbb S}^{n-1}}g\lambda_\varepsilon\,dS_{K}=\int_{{\mathbb S}^{n-1}}g(u)\varphi'_\varepsilon(h_K(u))\,d\mu(u)
$$
for any continuous function $g$ on ${\mathbb S}^{n-1}$.
Therefore $\lambda_\varepsilon\,dS_{K}=\varphi'_\varepsilon\circ h_K\,d\mu$.
 Q.E.D.

\section{Proof of Theorem~\ref{theodens}}
\label{sectheodens}

We start recalling that, by Corollary~\ref{Kepsilon}, $K^\varepsilon\subset \sigma(K^\varepsilon)+R_0B^n$
where $\sigma(K^\varepsilon)$ is the centroid and $R_0>1$ depends on $n$, $p$, $\tau_1$ and $\tau_2$.
The following lemma is a simple consequence of Lemma~\ref{centroid} (iii) and $V(K^\varepsilon)=1$.

\begin{lemma}
\label{r0}
For $r_0=\frac1{(n+1)R_0^{n-1}\kappa_{n-1}}$, we have $\sigma(K^\varepsilon)+r_0B^n\subset K^\varepsilon$.
\end{lemma}

Next we show that $\lambda_\varepsilon$ is bounded and bounded away from zero.

\begin{lemma}
\label{lambdaepsilonbound}
There exist $\tilde{\tau}_2>\tilde{\tau}_1>0$ depending on $n$, $p$, $\tau_1$ and $\tau_2$ such that
$\tilde{\tau}_1\leq\lambda_\varepsilon\leq\tilde{\tau}_2$ if $\varepsilon<\min\{\frac{r_0}6,\frac16\}$.
\end{lemma}
\proof We assume $\xi(K^\varepsilon)=o$.
To simplify the notation, we set $K=K^\varepsilon$ and $\sigma=\sigma(K)$. Let
$w\in {\mathbb S}^{n-1}$ and $\varrho\geq 0$ be such that $\sigma=\varrho w$. Since 
$r_0w\in K$, if $u\in {\mathbb S}^{n-1}$ and $\langle u,w\rangle\geq \frac12$, then
$h_K(u)\geq r_0/2$. Moreover, since $\varphi'_\varepsilon$ is monotone decreasing, we have
$\varphi'_\varepsilon(h_K(u))\geq \varphi'_\varepsilon(2R_0)=\varphi'(2R_0)$ for all $u\in {\mathbb S}^{n-1}$, and hence
(\ref{caparea}) yields
$$
\int_{{\mathbb S}^{n-1}}h_{K}(u)\cdot
\varphi'_\varepsilon(h_{K}(u))\,d\mu(u)\geq 
\int_{u\in {\mathbb S}^{n-1}\atop \langle u,w\rangle\geq \frac12}(r_0/2)\cdot
\varphi'(2R_0)\,d\mu(u)\geq (r_0/2)\cdot
\varphi'(2R_0)\tau_1 \cdot (\sqrt{3}/2)^{n-1}\kappa_{n-1},
$$
which in turn yields the required lower bound on $\lambda_\varepsilon$.

To have a suitable upper bound on $\lambda_\varepsilon$, the key observation is that
using $\varrho\leq R_0$, we deduce that
if $u\in {\mathbb S}^{n-1}$ with
$\langle u,w\rangle\geq -\frac{r_0}{2R_0}$ and $\varepsilon<\frac{r_0}6$ then
$$
h_K(u)\geq \langle u,\varrho w+r_0u\rangle\geq r_0-\frac{r_0\varrho}{2R_0}\geq r_0/2,
$$
therefore
\begin{equation}
\label{lambdaepsupperr}
\varphi'_\varepsilon(h_{K}(u))\leq \varphi'_\varepsilon(r_0/2)=\varphi'(r_0/2).
\end{equation}
Another observation is that $K\subset 2R_0B^n$ implies
\begin{equation}
\label{lambdaepsupperR}
h_K(u)<2R_0\mbox{ \ for any $u\in S^{n-1}$.}
\end{equation}
It follows directly from (\ref{lambdaepsupperr}) and (\ref{lambdaepsupperR}) that
\begin{equation}
\label{lambdaepsupper-}
\int_{u\in {\mathbb S}^{n-1}\atop \langle u,w\rangle\geq \frac{-r_0}{2R_0}}
 h_K(u)\varphi'_\varepsilon(h_{K}(u))\,d\mu(u)\leq  (2R_0)\varphi'(r_0/2) \tau_2n\kappa_n.
\end{equation}

However, if $\langle u,w\rangle< \frac{-r_0}{2R_0}$ for $u\in S^{n-1}$,  then 
$\varphi'_\varepsilon(h_{K^\varepsilon}(u))$ can be arbitrary large as $\xi(K^\varepsilon)$ can be arbitrary close to $\partial K^\varepsilon$ if $\varepsilon>0$ is small, and 
hence we transfer the problem to the case $\langle u,w\rangle\geq  \frac{-r_0}{2R_0}$ using
Corollary~\ref{intcond}. 
First we claim that
\begin{equation}
\label{lambdaepsupper}
\int_{u\in {\mathbb S}^{n-1}\atop \langle u,w\rangle< \frac{-r_0}{2R_0}}
\varphi'_\varepsilon(h_{K}(u))\,d\mu(u)\leq \frac{2R_0}{r_0}\cdot
\varphi'(r_0/2)\tau_2n\kappa_n.
\end{equation}
On the one hand, first applying Corollary~\ref{intcond}, and after that
$\mu({\mathbb S}^{n-1})\leq \tau_2n\kappa_n$
  and (\ref{lambdaepsupperr}) imply
$$
\int_{u\in {\mathbb S}^{n-1}\atop \langle u,w\rangle< \frac{-r_0}{2R_0}}
\langle u,-w\rangle \varphi'_\varepsilon(h_{K}(u))\,d\mu(u)=
\int_{u\in {\mathbb S}^{n-1}\atop \langle u,w\rangle\geq \frac{-r_0}{2R_0}}
\langle u,w\rangle \varphi'_\varepsilon(h_{K}(u))\,d\mu(u)\leq 
\varphi'(r_0/2)\tau_2n\kappa_n.
$$
On the other hand,
as $ \langle u,w\rangle< \frac{-r_0}{2R_0}$ is equivalent to 
$\langle u,-w\rangle> \frac{r_0}{2R_0}$,
we have
$$
\int_{u\in {\mathbb S}^{n-1}\atop \langle u,w\rangle< \frac{-r_0}{2R_0}}
\langle u,-w\rangle \varphi'_\varepsilon(h_{K}(u))\,d\mu(u)
\geq \frac{r_0}{2R_0}\int_{u\in {\mathbb S}^{n-1}\atop \langle u,w\rangle< \frac{-r_0}{2R_0}}
\varphi'_\varepsilon(h_{K}(u))\,d\mu(u),
$$
and in turn deduce (\ref{lambdaepsupper}).

Now (\ref{lambdaepsupperR}) and (\ref{lambdaepsupper}) yield
$$
\int_{u\in {\mathbb S}^{n-1}\atop \langle u,w\rangle< \frac{-r_0}{2R_0}}
h_K(u)\varphi'_\varepsilon(h_{K}(u))\,d\mu(u)\leq \frac{(2R_0)^2}{r_0}\cdot
\varphi'(r_0/2)\tau_2n\kappa_n,
$$
which estimate combined with (\ref{lambdaepsupper-}) leads to
$\lambda_\varepsilon< \frac{(2R_0)^2+2R_0}{r_0}\,\varphi'(r_0/2)\tau_2n\kappa_n$.
In turn, we conclude Lemma~\ref{lambdaepsilonbound}. Q.E.D. \\

\noindent{\it Proof of Theorem~\ref{theodens} } 
We assume that $\xi(K^\varepsilon)=o$ for all $\varepsilon\in(0,\min\{\frac16,\frac{r_0}6\})$.
 It follows from Lemma~\ref{Euler-Lagrange} that
\begin{equation}
\label{Euler-Lagrange0}
\varphi'_\varepsilon(h_{K^\varepsilon}(u))\,d\mu(u)=
\lambda_\varepsilon\,dS_{K^\varepsilon}
\end{equation}
as measures on ${\mathbb S}^{n-1}$.

Using the constants $r_0,R_0$ of Lemma~\ref{r0}, 
if $\varepsilon$ is small then $ K^\varepsilon\subset 2R_0B^n$ and 
$K^\varepsilon$ contains a ball of radius $r_0$. According to the Blaschke selection Theorem
and Lemma~\ref{lambdaepsilonbound}, there exists a sequence $\{\varepsilon_m\}$ tending to zero, $\varepsilon_m>0$, such that $K^{\varepsilon_m}$ tends to a convex body $K_0$, and 
$\lim_{m\to\infty}\lambda_{\varepsilon_m}=\lambda_0>0$. In particular, the surface area measure of $K^{\varepsilon_n}$
tends weakly to $S_{K_0}$, and we may assume that
\begin{equation}
\label{surface-area-estimate}
\lambda_{\varepsilon_m}S(K^{\varepsilon_m})\leq (\lambda_0+1)S(K)
\end{equation}
for all $m$. Here, for a convex body $K$, $S(K)$ denotes its surface area: $S(K)=S_K(\sfe)$.

We claim that the closed set $X=\{u\in {\mathbb S}^{n-1}:\,h_{K_0}(u)=0\}$ satisfies
\begin{equation}
\label{Xmeasure0}
\mu(X)=0.
\end{equation}
We may assume that $X\neq \emptyset$.
It follows from (\ref{phider}) that: setting $c=|p|$ if $p\in(-n,1)\backslash\{0\}$ and $c=1$ if $p=0$, we have
$$
\varphi'(t)= c\,t^{p-1}\mbox{ \ if $t\in(0,1)$.}
$$
Let $\tau\in (0,1)$. We can choose $m$ sufficiently large such that $3\varepsilon_m<\tau$
and $|h_{K^{\varepsilon_m}}(u)-h_{K_{0}}(u)|<\tau$ for $u\in {\mathbb S}^{n-1}$; thus,
if $0<t<\tau$, then
$$
\varphi'_{\varepsilon_m}(t)\geq \varphi'_{\varepsilon_m}(\tau)=\varphi'(\tau)=c\,\tau^{p-1}.
$$
In particular, 
$\varphi'_{\varepsilon_m}(h_{K^{\varepsilon_m}}(u))\geq c\,\tau^{p-1}$ holds for $u\in X$. It follows from
(\ref{Euler-Lagrange0}) and (\ref{surface-area-estimate}) that
$$
\mu(X)\leq \frac{(\lambda_0+1)S(K)}{c\,\tau^{p-1}}=\frac{(\lambda_0+1)S(K)}{c}\cdot \tau^{1-p}
$$
holds for any $\tau\in (0,1)$, and in turn we conclude (\ref{Xmeasure0}) as $1-p>0$.

Next, for $\delta\in(0,1)$, we define the closed set
$$
\Xi_\delta=\{u\in {\mathbb S}^{n-1}:\,h_{K_0}(u)\geq \delta\},
$$
so that ${\mathbb S}^{n-1}\backslash X=\cup_{\delta\in(0,1)}\Xi_\delta$.
For large $m$, we have $\varphi'_{\varepsilon_m}\circ h_{K^{\varepsilon_m}}=\varphi'\circ h_{K^{\varepsilon_m}}$
on $\Xi_\delta$, and the latter sequence tends uniformly to $\varphi'\circ h_{K_0}$ on $\Xi_\delta$. Therefore, if 
$g:\,{\mathbb S}^{n-1}\to\R$ is a continuous function, then (\ref{Euler-Lagrange0}) and the convergence of 
$K_{\varepsilon_m}$ to $K_0$ imply
$$
\int_{\Xi_\delta}g(u)\varphi'(h_{K_0}(u))\,d\mu(u)=\lambda_0\int_{\Xi_\delta}g(u)\,dS_{K_0}(u).
$$
We define
$$
\lambda=
\left\{\begin{array}{rl}
(\lambda_0/|p|)^{\frac1{n-p}}&\mbox{ if $p\in(-n,1)\backslash\{0\}$,}\\
\lambda_0^{\frac1{n-p}}&\mbox{ if $p=0$,}
\end{array}\right.
$$
and hence (\ref{phider})  yields
\begin{equation}
\label{gintphi}
\int_{\Xi_\delta}g(u)h_{K_0}(u)^{p-1}\,d\mu(u)=\lambda^{n-p}\int_{\Xi_\delta}g(u)\,dS_{K_0}(u).
\end{equation}
For any continuous $\psi:\,{\mathbb S}^{n-1}\to\R$,  $\psi(u)/h_{K_0}(u)^{p-1}$ is a continuous function on $\Xi_\delta$ that can be extended to a continuous function on ${\mathbb S}^{n-1}$. Using this function in place of $g$ in (\ref{gintphi}), we deduce that
$$
\int_{\Xi_\delta}\psi(u)\,d\mu(u)=\lambda^{n-p}\int_{\Xi_\delta}\psi(u)h_{K_0}(u)^{1-p}\,dS_{K_0}(u).
$$
As this holds for all $\delta\in(0,1)$, it follows that
\begin{equation}
\label{psiint}
\int_{{\mathbb S}^{n-1}\backslash X}\psi(u)\,d\mu(u)=\int_{{\mathbb S}^{n-1}\backslash X}\psi(u)h_{\lambda K_0}(u)^{1-p}\,dS_{\lambda K_0}(u).
\end{equation}
Combining (\ref{Xmeasure0}) and  (\ref{psiint}) implies that
$$
\int_{{\mathbb S}^{n-1}}\psi(u)\,d\mu(u)=\int_{\Xi_\delta}\psi(u)h_{\lambda K_0}(u)^{1-p}\,dS_{\lambda K_0}(u),
$$
for any continuous function $\psi:\,{\mathbb S}^{n-1}\to\R$, and hence $d\mu=h_{M}(u)^{1-p}\,dS_{M}(u)$
for $M=\lambda K_0$.  \ Q.E.D.\\

We still need to address the case when $\mu$ is invariant under certain closed subgroup $G$ of $O(n)$. Here the main additional difficulty is that 
we always have to deform the involved bodies in a $G$-invariant way.

\begin{prop}
\label{theodensG}
If $-n<p<1$ and the Borel measure $\mu$  satisfies
$d\mu =f\,d\mathcal{H}^{n-1}$ where $f$ is bounded, $\inf_{u\in {\mathbb S}^{n-1}}f(u)>0$ and
$f$ is invariant under the closed subgroup $G$ of $O(n)$, then there exists $M\in\mathcal{K}_0^n$ invariant under $G$ such that
$\mu=S_{M,p}$.
\end{prop}
 
To indicate the proof of  Proposition~\ref{theodensG}, we only sketch the necessary changes in the argument leading to Theorem~\ref{theodens}.

In this case, we consider the family $\mathcal{K}^G_1$ of convex bodies $K\in\mathcal{K}_1$ satisfying 
$AK=K$ for any $A\in G$. It follows from the uniqueness of $\xi(K)$ (see Proposition~\ref{xiinside}) that if $K\in\mathcal{K}^G_1$ and $A\in G$, then $A\xi(K)=\xi(K)$. 

The argument for Corollary~\ref{Kepsilon} carries over to yield  the following analogue statement.
For the $R_0>0$ depending on $n$, $p$, $\tau_1$ and $\tau_2$ of Lemma~\ref{Knotbounded01} and
Lemma~\ref{Knotboundedn1}, there exists
 $K^\varepsilon\in \mathcal{K}_1^G$  with $R(K^\varepsilon)\leq R_0$ for any $\varepsilon\in(0,\frac16)$ such that
$$
\Phi_\varepsilon(K^\varepsilon,\xi(K^\varepsilon))=\min_{K\in \mathcal{K}_1^G}\Phi_\varepsilon(K,\xi(K)).
$$

Let us discuss how to prove a $G$ invariant version of Corollary~\ref{quasi-smooth}; namely, that
$K^\varepsilon$ is quasi-smooth. In this case, a more subtle modification is needed.

\begin{lemma}
\label{quasi-smoothG}
$K^\varepsilon\in \mathcal{K}_1^G$ is quasi-smooth.
\end{lemma}
\proof We suppose that $K=K^\varepsilon\in\mathcal{K}_1^G$ is not quasi-smooth, and seek a contradiction.
We have $\mathcal{H}^{n-1}({\mathbb S}^{n-1}\backslash\nu_K(\partial' K))>0$, therefore there exists a
closed set $\tilde{\omega}\subset {\mathbb S}^{n-1}\backslash\nu_K(\partial' K)$ with $\mathcal{H}^{n-1}(\tilde{\omega})>0$.
We define
$$
\omega=\cup_{A\in G}A\tilde{\omega},
$$
which is compact as both $G$ and $\tilde{\omega}$ are compact. Readily, $\mathcal{H}^{n-1}(\tilde{\omega})>0$
and $\omega$ is $G$ invariant.
Since $K$ is $G$ invariant, we deduce that even
$\omega\subset {\mathbb S}^{n-1}\backslash\nu_K(\partial' K)$, and hence
 $S_K(\omega)=0$.
Thus  we can apply Lemma~\ref{improve}. We observe that  the set $K_t$ defined in Lemma~\ref{improve} is
now $G$ invariant,  and hence there exists
a  convex body $\widetilde{K}\in\mathcal{K}_1^G$ such that
$\Phi_\varepsilon(\widetilde{K},\xi(\widetilde{K}))<\Phi_\varepsilon(K,\xi(K))$. This contradiction with the extremality of $K=K^\varepsilon$ proves Lemma~\ref{quasi-smoothG}. Q.E.D. \\

Let us turn to the $G$-invariant version of Proposition~\ref{Euler-Lagrange}.

\begin{prop}
\label{Euler-LagrangeG}
$\varphi'_\varepsilon(h_{K^\varepsilon}(u)-\langle\xi(K^\varepsilon),u\rangle)\,d\mu(u)=
\lambda_\varepsilon\,dS_{K^\varepsilon}$ as measures on ${\mathbb S}^{n-1}$.
\end{prop}
\proof The key statement in the proof of Proposition~\ref{Euler-Lagrange}
is (\ref{support}), claiming that, if we assume $K=K^\varepsilon$ and  $\xi(K)=o$, for any convex body $C$ with 
$o\in{\rm int} C$ we have
\begin{equation}
\label{supportG}
\int_{{\mathbb S}^{n-1}}h_C\lambda_\varepsilon\,dS_{K}=\int_{{\mathbb S}^{n-1}}h_C(u)\varphi'_\varepsilon(h_K(u))\,d\mu(u).
\end{equation}
To prove (\ref{supportG}), we write $\vartheta_G$ to denote the $G$-invariant Haar probability measure on $\sfe$.
We define the $G$-invariant convex body $C_0$ by
$$
h_{C_0}=\int_Gh_{AC}\,d\vartheta_G(A).
$$
Running the proof of (\ref{support}), using $C_0$ in place of $C$, and observing
that 
$$
K_t=\{x\in K:\,\langle x,u\rangle\leq h_{K}(u)+th_{C_0}(u)\mbox{ \ \ for $u\in {\mathbb S}^{n-1}$}\}
$$
is $G$-invariant, we deduce that
\begin{equation}
\label{supportGC_0}
\int_{{\mathbb S}^{n-1}}h_{C_0}\lambda_\varepsilon\,dS_{K}=\int_{{\mathbb S}^{n-1}}h_{C_0}(u)\varphi'_\varepsilon(h_K(u))\,d\mu(u).
\end{equation}
Therefore the $G$-invariance of $K$ and $\mu$, the Fubini theorem and (\ref{supportGC_0}) imply that
\begin{eqnarray*}
\int_{{\mathbb S}^{n-1}}h_C\lambda_\varepsilon\,dS_{K}&=&
\int_G\int_{{\mathbb S}^{n-1}}h_{AC}\lambda_\varepsilon\,dS_{K}\,d\vartheta_G(A)\\
&=&\int_{{\mathbb S}^{n-1}}h_{C_0}\lambda_\varepsilon\,dS_{K}
=\int_{{\mathbb S}^{n-1}}h_{C_0}(u)\varphi'_\varepsilon(h_K(u))\,d\mu(u)\\
&=&\int_G\int_{{\mathbb S}^{n-1}}h_{AC}(u)\varphi'_\varepsilon(h_K(u))\,d\mu(u)\,d\vartheta_G(A)\\
&=&\int_{{\mathbb S}^{n-1}}h_C(u)\varphi'_\varepsilon(h_K(u))\,d\mu(u),
\end{eqnarray*}
yielding (\ref{supportG}). The rest of the proof of Proposition~\ref{Euler-Lagrange} carries over without any change.
 \ Q.E.D. \\

Having these tailored statements, the rest of the proof of Theorem~\ref{theodens} yields Proposition~\ref{theodensG}. 

The only part we do not prove here is that $o\in{\rm int}\, K$ when $p\leq -n+2$, which fact is verified using a simple argument by Chou and Wang \cite{CW}, and is also proved as Lemma~4.1 in \cite{BiBoCo}. \ Q.E.D.\\

\section{Some more simple facts needed to prove Theorems~\ref{theodens01} and \ref{theodens-n0}}

In order to prove Theorems~\ref{theodens01} and \ref{theodens-n0}, we continue our study using the same notation. However we now drop the assumption~\eqref{fcondition} on $f$, unless explicitly stated. The following is a simple consequence of the proof of
Theorem~\ref{theodens}.

\begin{lemma} 
\label{Phicentroid}
Let $p\in(-n,1)$ and $\mu$ be a measure on ${\mathbb S}^{n-1}$ with a bounded density function $f$ with respect to 
$\mathcal{H}^{n-1}$, such that ${\rm inf}\,f>0$; then there exists a convex body $M$ with
$o\in M$, 	$S_{M,p}=\mu$ and	
$$
\int_{{\mathbb S}^{n-1}}\varphi\left(V(M)^{\frac{-1}n}h_{M-\sigma(M)}(u)\right)\,d\mu
\leq\varphi(2\kappa_n^{-1/n})\mu({\mathbb S}^{n-1}).
$$
In addition, if $\mu$ is invariant under a closed subgroup $G$ of $O(n)$, then $M$ can be chosen to be invariant under $G$.
\end{lemma}
\proof  We recall that for any small $\varepsilon>0$, $K^\varepsilon\in\mathcal{K}_1$ satisfies
$$
\int_{{\mathbb S}^{n-1}}\varphi_\varepsilon\circ h_{K^\varepsilon-\xi(K^\varepsilon)}\,d\mu=
\min_{K\in\mathcal{K}_1}\max_{\xi\in {\rm int}\,K}
\int_{{\mathbb S}^{n-1}}\varphi_\varepsilon\circ h_{K-\xi}\,d\mu
$$
where $\xi(K^\varepsilon)\in{\rm int}\,K^\varepsilon$. In addition, if $\mu$ is invariant under the closed subgroup $G$ of $O(n)$, then $K^\varepsilon$ can be chosen to be invariant under $G$, and hence $\sigma(K^\varepsilon)$ is invariant under $G$, as well.
We deduce that (\ref{Phiball}) yields 
\begin{equation}
\label{Phiepsbounded}
\int_{{\mathbb S}^{n-1}}\varphi_\varepsilon\circ h_{K^\varepsilon-\sigma(K^\varepsilon)}\,d\mu\leq
\int_{{\mathbb S}^{n-1}}\varphi_\varepsilon\circ h_{K^\varepsilon-\xi(K^\varepsilon)}\,d\mu\leq\varphi(2\kappa_n^{-1/n})\mu({\mathbb S}^{n-1})
\end{equation} 
for any small $\varepsilon>0$. In the proof of Theorems~\ref{theodens} in Section~\ref{sectheodens}, we have proved
that there exist a sequence $\varepsilon_m$ with $\lim_{m\to\infty}\varepsilon_m=0$ 
and convex body $M$ with
$o\in M$ and $S_{M,p}=\mu$ such that 
$K^{\varepsilon_m}$ tends to some $\widetilde{K}\in\mathcal{K}_1$ where $\widetilde{K}=V(M)^{\frac{-1}n}\,M$. As $\sigma(K^{\varepsilon_m})$ tends  to $\sigma(\widetilde{K})$, we have that
$K^{\varepsilon_m}-\sigma(K^{\varepsilon_m})$ tends to $\widetilde{K}-\sigma(\widetilde{K})$. Therefore we conclude Lemma~\ref{Phicentroid} from $\sigma(\widetilde{K})\in{\rm int}\,\widetilde{K}$ and
(\ref{Phiepsbounded}). Q.E.D.\\

The following lemma bounds the inradius in terms of the $L_p$-surface area.

\begin{lemma}
\label{inradiusbounded}
Let $p<1$, and let $K$ be a convex body in $\R^n$ which contains $o$ and a ball of radius $r$, then
$$
S_{K,p}({\mathbb S}^{n-1})\geq \kappa_{n-1} r^{n-p}.
$$
\end{lemma}
\proof Let $x_0\in\R^n$ be such that $x_0+rB^n\subset K$. If  $x_0\neq o$ let $x_0=\theta v$ for $\theta> 0$ and $v\in {\mathbb S}^{n-1}$, otherwise let $v$ be any unit vector and let $\theta=0$. We define a subset of $\partial K$ as follows:
\[
\Xi=\{x\in\partial K:\,x=y+sv\mbox{ \ for $y\in r\left({\rm int}\,B^n\right)\cap v^\bot$ and $s>\theta$}\}.
\]
Let $x\in \Xi$, with  $x=y+sv$ for some $y\in r\left({\rm int}\,B^n\right)\cap v^\bot$ and $s>\theta$, and let $\nu_K(x)$ be an outer unit normal of $K$ at $x$. Since $x_0+r\nu_K(x)\in K$ and $x_0+y\in K$ we have
\begin{gather}
 \langle \nu_K(x),x_0+r\nu_K(x)-x\rangle\leq0, \label{inradius_bounded_a}\\
 \langle \nu_K(x),x_0+y-x\rangle\leq0. \label{inradius_bounded_b}
\end{gather}
Formula~\eqref{inradius_bounded_b} implies $\langle \nu_K(x),v\rangle\geq0$, and, as a consequence,
\begin{equation}\label{inradius_bounded_c}
 \langle \nu_K(x),x_0\rangle\geq0.
\end{equation}
Formula~\eqref{inradius_bounded_a} implies $\langle \nu_K(x),x\rangle\geq\langle \nu_K(x),x_0\rangle+r$, and, in view of \eqref{inradius_bounded_c},
\[
 \langle \nu_K(x),x\rangle\geq r.
\]
It follows from
$\mathcal{H}^{n-1}(\Xi)\geq \kappa_{n-1}r^{n-1}$ that
\[
S_{K,p}({\mathbb S}^{n-1})\geq \int_{\Xi}\langle \nu_K(x),x\rangle^{1-p}\,d\mathcal{H}^{n-1}(x)
\geq r^{1-p}\kappa_{n-1}r^{n-1},
\]
which proves Lemma~\ref{inradiusbounded}. \ Q.E.D.\\

\section{Proof of Theorem~\ref{theodens-n0}}

We have a non-trivial measure $\mu$ on ${\mathbb S}^{n-1}$ satisfying that $d\mu=f\,d\mathcal{H}^{n-1}$ for a non-negative $L_{\frac{n}{n+p}}$ function $f$. For any integer $m\geq 2$, we define $f_m$ on ${\mathbb S}^{n-1}$ as follows
$$
f_m(u)=\left\{
\begin{array}{ll}
m&\mbox{ \ if $f(u)\geq m$},\\[0.5ex]
f(u)&\mbox{ \ if $\frac1m<f(u)< m$},\\[0.5ex]
\frac1m&\mbox{ \ if $f(u)\leq\frac1m$}
\end{array}
\right. 
$$
and define the measure $\mu_m$ on ${\mathbb S}^{n-1}$ by $d\mu_m=f_m\,d\mathcal{H}^{n-1}$. 
Since $f$ is also in $L_1$ by H\"older's inequality, it follows from Lebesgue's Dominated Convergence theorem that $\mu_m$ tends weakly to $\mu$.
We choose $m_0$ such that
\begin{equation}
\label{-n0mummu}
\frac{\mu({\mathbb S}^{n-1}) }2<\mu_m({\mathbb S}^{n-1})<2\mu({\mathbb S}^{n-1})  \mbox{ \ for $m\geq m_0$}.
\end{equation}
According to Lemma~\ref{Phicentroid}, there exists a convex body $K_m$ with
$o\in K$, 	$S_{K_m,p}=\mu_m$ and	
\begin{eqnarray}
\label{-n0condition}
- V(K_m)^{\frac{|p|}n}\int_{{\mathbb S}^{n-1}}h_{K_m-\sigma(K_m)}^p\,d\mu_m&=&
\int_{{\mathbb S}^{n-1}}-\left(V(K_m)^{\frac{-1}n}h_{K_m-\sigma(K_m)}\right)^p\,d\mu_m\\
\noindent
&\leq& -(2\kappa_n^{-1/n})^p\mu_m({\mathbb S}^{n-1})
\leq-\frac{(2\kappa_n^{-1/n})^p}2\cdot \mu({\mathbb S}^{n-1}).
\end{eqnarray}
In addition, if $\mu$ is invariant under the closed subgroup $G$ of $O(n)$, then 
each $\mu_m$ is invariant under $G$, and hence 
$K_m$ can be chosen to be invariant under $G$.

\begin{lemma}
\label{Pmbounded-n0}
$\{K_m\}$ is bounded.
\end{lemma}
\proof We set 
\begin{eqnarray*}
\varrho_m&=&\max\{\varrho:\,\sigma(K_m)+\varrho\,B^n\subset K_m\}\\
R_m&=&\min\{\|x-\sigma(K_m)\|:\,x\in K_m\}\\
t_m&=&\min\left\{\frac12,R_m^{\frac{-1}{2n}}\right\},
\end{eqnarray*}
choose $v_m\in {\mathbb S}^{n-1}$ such that $\sigma(K_m)+R_mv_m\in\partial K_m$, and define
$$
\Xi_m=\{u\in {\mathbb S}^{n-1}:\,|\langle u,v_m\rangle|\leq t_m\}.
$$
Lemma~\ref{inradiusbounded} and \eqref{-n0mummu} imply
\[
 \varrho_m\leq \left( \frac{S_{K,p}({\mathbb S}^{n-1})}{\kappa_{n-1}}\right)^\frac1{n-p}
 \leq \left( \frac{2\mu({\mathbb S}^{n-1})}{\kappa_{n-1}}\right)^\frac1{n-p}.
\]
Thus, by Lemma~\ref{centroid} (iii), we have
\begin{equation}
\label{VKmRm}
V(K_m)\leq (n+1)\kappa_{n-1}\varrho_mR_m^{n-1}\leq
(n+1)\kappa_{n-1}\left(\frac{2\mu({\mathbb S}^{n-1})}{\kappa_{n-1}}\right)^{\frac1{n-p}}R_m^{n-1}
\leq c_0 R_m^{n-1}
\end{equation}
for a $c_0>0$ depending on $\mu,n,p$.

We suppose that $\{K_m\}$ is unbounded, thus there exists a subsequence $\{R_{m'}\}$  of
$\{R_m\}$ tending to infinity, and seek a contradiction.
We may assume that $\{v_{m'}\}$ tends to $v\in {\mathbb S}^{n-1}$. In addition, the definition of $t_m$ yields
\begin{equation}
\label{tmzero}
\lim_{m'\to\infty}t_{m'}=0.
\end{equation}
We claim that
\begin{equation}
\label{Ximtozero}
\lim_{m'\to\infty}\int_{\Xi_{m'}}f^{\frac{n}{n-|p|}}\,d\mathcal{H}^{n-1}=0,
\end{equation}
which is equivalent to show that the left hand side in (\ref{Ximtozero}) is at most $\tau$ for any small 
$\tau>0$. For $s\in(0,1)$, we set
$$
\widetilde{\Xi}(s)=\{u\in {\mathbb S}^{n-1}:\,\langle u,v\rangle\leq s\}.
$$
Since $f$ is in $L_{\frac{n}{n+p}}$ with respect to $\mathcal{H}^{n-1}$,   there exists 
$\delta\in(0,\frac12)$ such that
\begin{equation}
\label{Ximtozerotau}
\int_{\widetilde{\Xi}(2\delta)}f^{\frac{n}{n-|p|}}\,d\mathcal{H}^{n-1}<\tau.
\end{equation}
 Now if $m'$ is large, then
$t_{m'}<\delta$ by (\ref{tmzero}),
and  hence $\Xi_{m'}\subset \widetilde{\Xi}(2\delta)$ as $v_{m'}$ tends to $v$.
Therefore (\ref{Ximtozerotau}) implies (\ref{Ximtozero}).

Next we claim
\begin{equation}
\label{intonXim}
\lim_{m'\to\infty}
V(K_{m'})^{\frac{|p|}n}\int_{\Xi_{m'}}h_{K_{m'}-\sigma(K_{m'})}^p\,d\mu=0.
\end{equation}
We deduce from the H\"older inequality and the form of the Blaschke-Santal\'o inequality given in Lemma~\ref{centroid} (ii) 
\begin{eqnarray*}
\int_{\Xi_{m'}}h_{K_{m'}-\sigma(K_{m'})}^p\,d\mu&= &
\int_{\Xi_{m'}}h_{K_{m'}-\sigma(K_{m'})}^{-|p|}f\,d\mathcal{H}^{n-1}\\
&\leq & \left(\int_{\Xi_{m'}}h_{K_{m'}-\sigma(K_{m'})}^{-n}\,d\mathcal{H}^{n-1}\right)^{\frac{|p|}n}
\left(\int_{\Xi_{m'}}f^{\frac{n}{n-|p|}}\,d\mathcal{H}^{n-1}\right)^{\frac{n-|p|}n}\\
&\leq&\kappa_n^{\frac{2|p|}n}\, n^{\frac{|p|}{n}}\,V(K_{m'})^{\frac{-|p|}n}
\left(\int_{\Xi_{m'}}f^{\frac{n}{n-|p|}}\,d\mathcal{H}^{n-1}\right)^{\frac{n-|p|}n}.
\end{eqnarray*}
In turn, (\ref{Ximtozero}) yields (\ref{intonXim}).

We also prove
\begin{equation}
\label{intoutXim}
\lim_{m'\to\infty}
V(K_{m'})^{\frac{|p|}n}\int_{{\mathbb S}^{n-1}\backslash \Xi_{m'}}h_{K_{m'}-\sigma(K_{m'})}^p\,d\mu=0.
\end{equation}
We observe that if $u\in {\mathbb S}^{n-1}\backslash \Xi_{m'}$, then $|\langle u,v_{m'}\rangle|>t_{m'}$. Since
$\sigma(K_{m'})-\frac{R_{m'}}n\,v_{m'}\in K$ according to Lemma~\ref{centroid} (i), we deduce that
$$
h_{K_{m'}-\sigma(K_{m'})}(u)\geq \max\left\{\left\langle u,-\frac{R_{m'}}n\,v_{m'}\right\rangle,
\langle u,R_{m'}\,v_{m'}\rangle  \right\}\geq \frac{R_{m'}t_{m'}}n.
$$
It follows, by~\eqref{VKmRm} and the definition of $t_{m'}$, that
$$
V(K_{m'})^{\frac{|p|}n}\int_{{\mathbb S}^{n-1}\backslash \Xi_{m'}}h_{K_{m'}-\sigma(K_{m'})}^p\,d\mu\leq
n^{|p|}c_0^{\frac{|p|}n}R_{m'}^{\frac{|p|(n-1)}n}(R_{m'}t_{m'})^{-|p|}\mu({\mathbb S}^{n-1})=
n^{|p|}c_0^{\frac{|p|}n}\mu({\mathbb S}^{n-1})R_{m'}^{\frac{-|p|}{2n}}
$$
proving (\ref{intoutXim}).

We deduce from (\ref{intonXim}) and (\ref{intoutXim}) that
$$
\lim_{m'\to\infty}
V(K_{m'})^{\frac{|p|}n}\int_{{\mathbb S}^{n-1}}h_{K_{m'}-\sigma(K_{m'})}^p\,d\mu=0,
$$
contradicting (\ref{-n0condition}), and proving Lemma~\ref{Pmbounded-n0}. Q.E.D.\\

\emph{Proof of Theorem~\ref{theodens-n0}.} It follows from Lemma~\ref{Pmbounded-n0} that there is a subsequence $\{K_{m'}\}$ of 
$\{K_m\}$ that tends to a compact convex set $K_0$. Since $S_{K_{m'},p}$ tends weakly to
$S_{K_0,p}$, we deduce that $\mu=S_{K_0,p}$. Since $S_{K,p}$ is the null measure when $p<1$ and $K$ has empty interior, we deduce that  ${\rm int}\,K_0\neq \emptyset$.
We note that if $\mu$ is invariant under the closed subgroup $G$ of $O(n)$, then 
$K_0$ is invariant under $G$.
Q.E.D.

\section{Proof of Theorem~\ref{theodens01} when any open hemisphere has positive measure}
\label{sec01lp1}

Let $p\in(0,1)$, and let $\mu$ be a non-trivial measure on ${\mathbb S}^{n-1}$ such that that any open hemisphere of ${\mathbb S}^{n-1}$ has positive measure. In addition, we assume that $\mu$ is invariant under the closed subgroup $G$ of $O(n)$ (possibly $G$ is a trivial subgroup).  For a finite set $Z$, we write $\# Z$ to denote its cardinality.

First we construct a sequence $\{\mu_m\}$ of $G$ invariant Borel measures weakly approximating $\mu$. 
For any $u\in {\mathbb S}^{n-1}$, we write $\Gamma_u=\{Au:\,A\in G\}$ to denote its orbit. The space of orbits is $X={\mathbb S}^{n-1}/\sim$ where $u\sim v$ if and only if $v=Au$ for some $A\in G$; let $\psi:\,{\mathbb S}^{n-1}\to X$ be the quotient map.
Since $G$ is compact, $X$ is a metric space with the metric
$$
d(\psi(u),\psi(v))=\min\{\angle (y,z):\, y\in\Gamma_u\mbox{ \ and \ }z\in\Gamma_v\}.
$$
For $m\geq 2$, let $x_1,\ldots,x_k\in X$ be an $1/m$-net; namely, for any $x\in X$, there exists $x_i$ with 
$d(x,x_i)\leq 1/m$. For any $x_i$, $i=1,\ldots,k$, we consider its Dirichlet-Voronoi cell
$$
D_i=\{x\in X:\,d(x,x_i)\leq d(x,x_j) \mbox{ \ for $j=1,\ldots,k$}\},
$$
and hence $d(x,x_i)\leq 1/m$ for $x\in D_i$. 
We set $U_0=\emptyset$ and, for $i=1,\ldots,k-1$, we define
$$
U_i=\bigcup\{\psi^{-1}(D_j):\,j=1,\ldots,i\}.
$$
We subdivide ${\mathbb S}^{n-1}$ into the pairwise disjoint Borel sets
$$
\mathcal{D}_m=\{\psi^{-1}(D_i)\backslash U_{i-1}:\, \mbox{$i=1,\ldots,k$}\}
$$
where each $\Pi\in \mathcal{D}_m$  satisfies that $\Pi$ is $G$ invariant, $\mathcal{H}^{n-1}(\Pi)>0$ and for
any $u\in\Pi$, there exists $A\in G$ with $\angle(Au,z(\Pi))\leq 1/m$ for a fixed $z(\Pi)\in\Pi$ with
$\psi(z(\Pi))\in\{x_1,\ldots,x_k\}$. 

It is time to define the density function for $\mu_m$ by
$$
f_m(u)=\frac{\mu(\Pi)}{\mathcal{H}^{n-1}(\Pi)}+\frac1{(\# \mathcal{D}_m)^2} 
\mbox{ \ \ \ \ if $u\in\Pi$ and $\Pi\in \mathcal{D}_m$},
$$
in other words, $d\mu_m=f_m\,d\mathcal{H}^{n-1}$. It follows that each $\mu_m$ is invariant under $G$,
each $f_m$ is bounded with $\inf_{u\in {\mathbb S}^{n-1}}f_m(u)>0$.

Let us show that the sequence
$\{\mu_m\}$ tends weakly to $\mu$. For any continuous $g:\,{\mathbb S}^{n-1}\to\R$, we define
the $G$ invariant function $g_0:\,{\mathbb S}^{n-1}\to\R$ by
$$
g_0(u)=\int_G g(Au)\,d\vartheta_G(A)
$$
where $\vartheta_G$ is the invariant Haar probability measure on $G$. Since $\mu$ is $G$ invariant, the Fubini theorem yields
$$
\int_{{\mathbb S}^{n-1}}g\,d\mu=\int_{{\mathbb S}^{n-1}}g_0\,d\mu\mbox{ \ and \ }
\int_{{\mathbb S}^{n-1}}g\,d\mu_m=\int_{{\mathbb S}^{n-1}}g_0\,d\mu_m
$$
for $m\geq 2$. The construction of $\mathcal{D}_m$ implies that
$\lim_{m\to\infty}\int_{{\mathbb S}^{n-1}}g_0\,d\mu_m=\int_{{\mathbb S}^{n-1}}g_0\,d\mu$, and hence
$\{\mu_m\}$ tends weakly to $\mu$.

We may assume that $m_0$ is large enough to ensure that
\begin{equation}
\label{01mummu}
\mu_m({\mathbb S}^{n-1})<2\mu({\mathbb S}^{n-1}) \mbox{ \ for $m\geq m_0$}.
\end{equation}
According to Lemma~\ref{Phicentroid}, there exists a convex body $K_m$ with
$o\in K_m$, 	$S_{K_m,p}=\mu_m$ and	
\begin{eqnarray}
\label{01condition}
 V(K_m)^{\frac{-p}n}\int_{{\mathbb S}^{n-1}}h_{K_m-\sigma(K_m)}^p\,d\mu_m&=&
\int_{{\mathbb S}^{n-1}}\left(V(K_m)^{\frac{-1}n}h_{K_m-\sigma(K_m)}\right)^p\,d\mu_m\\
\nonumber
&\leq &(2\kappa_n^{-1/n})^p\mu_m({\mathbb S}^{n-1})
\leq 2(2\kappa_n^{-1/n})^p\mu({\mathbb S}^{n-1}).
\end{eqnarray}
In addition, each $K_m$ can be chosen to be invariant under $G$.

\begin{lemma}
\label{Pmbounded01}
$\{K_m\}$ is bounded.
\end{lemma}
\proof For $m\geq m_0$, we set
\begin{eqnarray*}
\varrho_m&=&\max\{\varrho:\,\sigma(K_m)+\varrho\,B^n \subset K_m\}\\
R_m&=&\min\{\|x-\sigma(K_m)\|:\,x\in K_m\},
\end{eqnarray*}
and choose $v_m\in {\mathbb S}^{n-1}$ such that $\sigma(K_m)+R_mv_m\in\partial K_m$.
It follows from Lemma~\ref{centroid} (iii), Lemma~\ref{inradiusbounded} and (\ref{01mummu}) that
\begin{equation}
\label{VKmRm01}
V(K_m)\leq (n+1)\kappa_{n-1}\varrho_mR_m^{n-1}\leq
(n+1)\kappa_{n-1}\left(\frac{2\mu({\mathbb S}^{n-1})}{\kappa_{n-1}}\right)^{\frac1{n-p}}R_m^{n-1}
\leq c_0 R_m^{n-1}
\end{equation}
for a $c_0>0$ depending on $\mu,n,p$.

We suppose that $\{K_m\}$ is unbounded, thus there exists a subsequence $\{R_{m'}\}$  of
$\{R_m\}$ tending to infinity, and seek a contradiction.
We may assume that $\{v_{m'}\}$ tends to $v\in {\mathbb S}^{n-1}$. 

For $w\in {\mathbb S}^{n-1}$ and $\alpha\in(0,\frac{\pi}2]$, we recall that $\Omega(w,\alpha)$ is the family of all $u\in {\mathbb S}^{n-1}$ with $\angle(u,w)\leq\alpha$. Since the $\mu$ measure of the open hemisphere centered at $v$ is positive, there exists $\delta>0$ and $\gamma\in(0,\frac{\pi}6)$ such that $\mu(\Omega(v,\frac{\pi}2-3\gamma))>2\delta$. 
As $\mu_m$ tends to $\mu$ weakly, there exists
$m_1\geq m_0$ such that if $m'\geq m_1$, then $\mu_{m'}(\Omega(v,\frac{\pi}2-2\gamma))>\delta$
and $\angle(v_{m'},v)<\gamma$. Therefore if $m'\geq m_1$, then
$$
\mu_{m'}\left(\Omega\left(v_{m'},\frac{\pi}2-\gamma\right)\right)>\delta.
$$
If $u\in \Omega(v_m,\frac{\pi}2-\gamma)$ then $\langle u,v_m\rangle\geq \sin\gamma$. Therefore $h_{K_{m'}-\sigma(K_{m'})}(u)\geq R_{m'}\sin\gamma$ and  
$$
\int_{\Omega(v_{m'},\frac{\pi}2-\gamma)}h_{K_{m'}-\sigma(K_{m'})}^p\,d\mu_{m'}\geq 
 (R_{m'}\sin\gamma)^p \delta.
$$
Inequality (\ref{VKmRm01}) yields
$$
\lim_{m'\to \infty}V(K_{m'})^{\frac{-p}n}\int_{{\mathbb S}^{n-1}}h_{K_{m'}-\sigma(K_{m'})}^p\,d\mu_{m'}\geq
\lim_{m'\to \infty}c_0^{\frac{-p}n}R_{m'}^{\frac{-p(n-1)}n}\cdot (R_{m'}\sin\gamma)^p \delta=\infty.
$$
This contradicts (\ref{01condition}), and proves Lemma~\ref{Pmbounded01}. \ Q.E.D.\\

\emph{Proof of Theorem 1.3 under the assumption that $\mu(\Sigma)>0$, for each open hemisphere $\Sigma$ of ${\mathbb S}^{n-1}$.} It follows from Lemma~\ref{Pmbounded01} that there is a subsequence $\{K_{m'}\}$ of 
$\{K_m\}$ that tends to a compact convex set $K_0$. Since $S_{K_{m'},p}$ tends weakly to
$S_{K_0,p}$, we deduce that $\mu=S_{K_0,p}$ and ${\rm int}\,K_0\neq \emptyset$. 
We note that if $\mu$ is invariant under the closed subgroup $G$ of $O(n)$, then 
$K_0$ is invariant under $G$.
Q.E.D.

\section{Proof of Theorem~\ref{theodens01} when the  measure is concentrated on a closed hemisphere}
\label{sec01lp2}

Let $p\in(0,1)$. First we show that the assumption required  in Conjecture~\ref{conjdens01} is necessary.

\begin{lemma}
\label{01lpnotantipodal}
If $p<1$ and $K\in \mathcal{K}_0$, then $\supp\,S_{K,p}$ is not a pair of antipodal points. 
\end{lemma}
\proof We  suppose  that $\supp\,S_{K,p}=\{w,-w\}$ for some $w\in {\mathbb S}^{n-1}$, and seek a contradiction.
Since the surface area measure of any open hemi-sphere is positive,
we have $o\in\partial K$. Let $\sigma$ be the exterior normal cone at $o$; namely,
$$
\sigma=\{y\in\R^n:\,\langle x,y\rangle\leq 0\;\forall x\in K\}=
\{y\in\R^n:h_K(y)=0\}.
$$
It follows that $w,-w\not\in \sigma$ by $p<1$, therefore the orthogonal projection $\sigma'$ of $\sigma$ into $w^\bot$ does not contain the origin in its interior. We deduce from the Hanh-Banach theorem the existence of a $(n-2)$-dimensional linear subspace 
$L_0\subset w^\bot$ supporting $\sigma'$. Therefore the $(n-1)$-dimensional linear subspace $L=L_0+\R w$ is a supporting hyperplane to $\sigma$ at $o$.  
We write $L^+$ to denote the open halfspace determined by $L$ not containing $\sigma$. We have $S_K(L^+\cap {\mathbb S}^{n-1})>0$ on the one hand, and $h_K(u)>0$ if $u\in L^+\cap {\mathbb S}^{n-1}$ on the other hand. We deduce that
$$
S_{K,p}(L^+\cap {\mathbb S}^{n-1})=\int_{L^+\cap {\mathbb S}^{n-1}}h_K^{1-p}\,dS_K>0.
$$
In particular, $\supp\,S_{K,p}\cap (L^+\cap {\mathbb S}^{n-1})\neq \emptyset$, contradicting 
$\supp\,S_{K,p}=\{w,-w\}$. \ Q.E.D.\\

We remark that $\supp\,S_{K,p}$ can consist of a single point, as the example of a pyramid with apex at $o$ shows.

Now we prove a sufficient condition ensuring that a measure $\mu$ on ${\mathbb S}^{n-1}$ is an $L_p$-surface area measure.
For any closed convex set $X\subset \R^n$, we write ${\rm relint}\,X$ to denote the interior of $X$ with respect to 
${\rm aff}\,X$. 

\emph{Completion of the proof of Theorem~\ref{theodens01}}. 
The idea is that we associate a measure $\mu_0$ on ${\mathbb S}^{n-1}$ to $\mu$ such that the $\mu_0$ measure of any open hemisphere is positive, construct a convex body $K_0$ whose 
$L_p$-surface area measure is $\mu_0$, and then take a suitable section of $K_0$. 

Let 
$C={\rm pos}\,\supp\,\mu$ and $L={\rm lin}\,\supp\,\mu$, and let
$v_0\in{\rm relint}\,C\cap {\mathbb S}^{n-1}$.
For
$$
\sigma=\{y\in L\, \langle y,v\rangle\leq 0\mbox{ for }v\in C\},
$$
the condition $L\neq C$ yields that $\sigma\cap L\neq \{o\}$. 

We claim that $(-\sigma)\cap {\rm relint}\,C\neq\emptyset$. If it didn't hold, then
the Hahn-Banach theorem applied to $C$ and $\sigma$ yields a 
$w\in S^{n-1}\cap L$ such that 
$\langle w,x\rangle\leq 0$ for $x\in C$, and
$\langle w,y\rangle\geq 0$ for $y\in -\sigma$. In particular, $w\in \sigma$, and as $y=-w\in \sigma$, we have
$$
-1=\langle w,y\rangle\geq 0.
$$
This contradiction proves that there exists a 
$v_0\in(-\sigma)\cap {\rm relint}\,C\cap S^{n-1}$.
In particular, we have 
\begin{equation}
\label{v0L}
\mbox{$\langle u,v_0\rangle\geq 0$ for all $u\in \supp\,\mu$.}
\end{equation}

We write $\widetilde{L}=L\cap v_0^\bot$, and set $d=n-{\rm dim}\,\widetilde{L}$ where 
$1\leq d\leq n$. 
We observe that $\supp\,\mu$ is contained in the half space of $L$ bounded by $\widetilde{L}$ and containing $v_0$ by (\ref{v0L}).
We consider
a $d$-dimensional regular simplex $S_0$ in $\widetilde{L}^\bot$ with vertices 
$v_0,\ldots,v_d\in {\mathbb S}^{n-1}\cap \widetilde{L}^\bot$,
and the $A\in O(n)$ that acts as the identity map on $\widetilde{L}$, and satisfies $Av_i=v_{i+1}$ for $i=0,\ldots,d-1$. 
We consider  the cyclic group $G_0$  of the isometries of $S_0$ of order $d+1$ generated by $A$, and the subgroup $\widetilde{G}$ of $O(n)$ generated by $G$ and $G_0$.
We define the Borel measure $\mu_0$ invariant under $\widetilde{G}$ in a way such that if $\omega\subset {\mathbb S}^{n-1}$ is Borel, then
$$
\mu_0(\omega)=\sum_{i=0}^d\mu(A^i\omega).
$$
In particular, $\supp\,\mu_0=\cup_{i=0}^dA^i\supp\,\mu$.

We prove that for any $w\in {\mathbb S}^{n-1}$, there exists
\begin{equation}
\label{notinopenhemisphere}
u\in \supp\,\mu_0\mbox{ \ such that \ }\langle w,u\rangle >0.  
\end{equation}
Since $v_0+\ldots+v_d=0$,  
either there exists $i\in\{0,\ldots,d\}$ such that $\langle w,v_i\rangle >0$, or 
 $w\in \widetilde{L}$, and hence $\langle w,v_i\rangle=0$. 
For $L_i={\rm lin}\{v_i,\widetilde{L}\}=A^iL$, we write $w=w_i+\tilde{w}_i$ where
$w_i\in L_i$ and $\tilde{w}_i\in L_i^\perp$, and hence  either
$\langle w_i,v_i\rangle >0$, or $w_i=w\in \widetilde{L}$, which in turn also yield that 
$w_i\neq 0$.
Since $v_i\in{\rm relint}\,A^iC$, there exists $u\in A^i\supp\,\mu$
with $\langle w_i,u\rangle >0$, and hence $\langle w_i,u\rangle> 0$.
In turn, we conclude (\ref{notinopenhemisphere}), therefore
the $\mu_0$ measure of any open hemisphere of 
${\mathbb S}^{n-1}$ is positive.

Now the argument in Section~\ref{sec01lp1} provides a convex body $K_0\in\mathcal{K}_0^n$
whose $L_p$-surface area is $\mu_0$  and is invariant under $\widetilde{G}$.
For $i=0,\ldots,d$, the Dirichlet-Voronoi cell of $v_i$ is defined by
$$
D(v_i)=\{x\in\R^n:\,\langle x,v_i\rangle\geq \langle x,v_j\rangle\mbox{ \ for $j=0,\ldots,d$}\},
$$
which is a polyhedral cone with $v_i\in{\rm int}\,D(v_i)$. Readily, $AD(v_i)=D(v_{i+1})$ for
$i=0,\ldots,d-1$ and $\R^n=\bigcup_{j=0}^d A^jD(v_0)$, where the sets in the union have disjoint interiors.

We define
\[
K=K_0\cap D(v_0)
\]
and prove that $S_p(K,\omega)=\mu(\omega)$ for each Borel set $\omega\subset \sfe$.
Let 
\[
N=\bigcup_{x\in\inte D(v_0)}\nu_K(x)=\bigcup_{x\in\inte D(v_0)}\nu_{K_0}(x).
\]
First we observe that 
\begin{equation}\label{th13_formula1}
S_p(K,\omega)=S_p(K,\omega\cap N).
\end{equation}
Indeed, if $u\notin N$ then either $u\in\nu_K(o)$ and, as a consequence, $h_K(u)=0$, or $u\in\nu_K(x)$ for some $x$ in the intersection of $\partial D(v_0)$ and of the closure of $(\partial K)\cap \inte D(v_0)$, an intersection whose $(n-1)$-dimensional Hausdorff measure is zero. These facts imply $S_p(K,\omega\setminus N)=0$ and \eqref{th13_formula1}. 

Then we prove that if $u\in\supp\,\mu_0\setminus\tilde L$ and $u\in\nu_{K_0}(x)$ for some $x\in\partial K_0\setminus D(v_j)$ then 
\begin{equation}\label{th13_formula5}
u\notin A^j\supp\,\mu.
\end{equation}
We prove \eqref{th13_formula5} for $j=0$ arguing by contradiction; the other cases can be proved similarly. Assume that $u\in\supp\, \mu$. Since $x\notin D(v_0)$ we have that $x\in D(v_i)\setminus D(v_0)$, for some $i\in\{1,\dots,d\}$, that is $\langle x,v_0\rangle<\langle x,v_i\rangle$.  The symmetries of $K_0$ imply that $x=A^{i}y$ for some $y\in K_0$. The inclusion $\supp\,\mu\subset C$ and  \eqref{v0L} imply  $u=\alpha v_0+p$  for some $\alpha> 0$ and $p\in \widetilde{L}$.
It follows that
\begin{eqnarray*}
\langle y,u\rangle&=&\alpha\langle y,v_0\rangle+\langle y,p\rangle=
\alpha\langle A^{i}y, A^{i}v_0\rangle+\langle A^{i} y, A^{i}p\rangle=
\alpha\langle x,v_i\rangle+\langle x,p\rangle\\
&>&
\alpha\langle x,v_0\rangle+\langle x,p\rangle=\langle x,u\rangle.
\end{eqnarray*}
This contradicts the fact that $u$ is an exterior unit normal at $x$ to $\partial K_0$ and conclude the proof of \eqref{th13_formula5}.
The previous claim easily implies
\begin{equation}\label{th13_formula2}
N\cap\supp\,\mu_0\subset\supp\,\mu\quad\text{ and}\quad\nu_{K_0}^{-1}(N\cap\supp\,\mu_0\setminus\tilde L)\subset D(v_0).
\end{equation}

Formulas~\eqref{th13_formula2} imply
\begin{equation}\label{th13_formula6}
S_p(K, \omega\cap N\setminus\tilde L)=S_p(K_0, \omega\cap N\setminus\tilde L)=\mu(\omega\cap N\setminus\tilde L).
\end{equation}

On the other hand, if $u\in\tilde L$ then $A^i\nu_{K_0}^{-1}(u)=\nu_{K_0}^{-1}(u)$, for each $i$, and
\begin{align*}
\nu_{K_0}^{-1}(u)&=\bigcup_{i=0}^d \nu_{K_0}^{-1}(u)\cap A^i D(v_0)
=\bigcup_{i=0}^d A^i\Big(\nu_{K_0}^{-1}(u)\cap  D(v_0)\Big)
=\bigcup_{i=0}^dA^i\Big( \nu_{K}^{-1}(u)\Big),
\end{align*}
where the sets in the last union have disjoint relative interiors. Moreover $h_{K_0}(u)=h_K(u)$. Thus
\begin{equation}\label{th13_formula7}\begin{aligned}
S_p(K,\omega\cap N\cap\tilde L)
=&\int_{\nu_{K}^{-1}(\omega\cap N\cap\tilde L)}\langle x,\nu_K(x)\rangle^{1-p} d\mathcal{H}^{n-1}(x)\\
=&\frac1{d+1}\int_{\nu_{K_0}^{-1}(\omega\cap N\cap\tilde L)}\langle x,\nu_{K_0}(x)\rangle^{1-p} d\mathcal{H}^{n-1}(x)\\
=&\frac1{d+1}\mu_0\big(\omega\cap N\cap\tilde L\big)\\
=&\mu\big(\omega\cap N\cap\tilde L\big)
\end{aligned}\end{equation}
Formulas~\eqref{th13_formula1}, \eqref{th13_formula6} and \eqref{th13_formula7} imply that $S_p(K,\omega)=\mu(\omega)$, or in other words, that $\mu$ is the $L_p$-surface area measure of $K$.  Q.E.D.\\

\begin{example}
If $L\subset\R^n$ is a linear $d$-subspace with $2\leq d\leq n-1$, then 
there exists a convex body $K$ such that $L={\rm pos}\,\supp\,\mu$
for the $L_p$-surface area measure of $K$. To construct such a $K$, we take a $d$-ball $B\subset L$ such that $o\in\partial B$, and the exterior unit normal $v$ to $B$ at $o$. We also consider an $(n-d+1)$-dimensional convex cone $\sigma\subset{\rm lin}\{L^\bot,v\}$ 
with $v\in{\rm relint}\sigma$ and $\langle v,w\rangle>0$ for $w\in \sigma\backslash\{o\}$.  We define $K$ with the formula
$$
K=\{x\in B+L^\bot:\,\langle x,y\rangle\leq 0 \mbox{ for } y\in\sigma \}.
$$ 
\end{example}

\section{The critical case $p=-n$}
\label{secCritical}

Let $K\in{\mathcal K}_{0}^n$ with $o\in{\rm int}\,K$ and $\partial K$ is $C^3_+$, and hence
$$
dS_{K,-n}=f\,d{\mathcal H}^{n-1}
$$ 
for a $C^1$ function $f(u)=h_K(u)^{n+1}/\kappa(u)$ on ${\mathbb S}^{n-1}$ (see \eqref{SK-n}), where $\kappa(u)$ is the Gaussian curvature at $x\in\partial K$ with $\nu_K(x)=u$. For basic notions in this section, we refer to Schneider \cite{SCH} and Yang \cite{Yan10}.

Let $h=h_K$, and let $\tilde{h}=h_{K^*}$ be the support function of the polar body $K^*$, defined as follows: 
$$
K^*=\{x\in\R^n:\,\langle x,y\rangle\leq 1 \,\forall y\in K\}. 
$$
In particular, $h_{K*}(u)^{-1}u\in\partial K$ for $u\in {\mathbb S}^{n-1}$, and
both $h$ and $\tilde{h}$ are $C^2$ on $\R^n\backslash \{o\}$. We write $\tilde{f}$ to denote 
the curvature function on $\R^n$, that is the $(-n-1)$ homogeneous function satisfying
 $\tilde{f}(u)=\kappa(u)^{-1}$ for $u\in {\mathbb S}^{n-1}$. 

We also recall some definitions and results from \cite{Yan10}.
Given a function $\phi: \R^n\backslash\{o\}\rightarrow \R$,
  let $\nabla\phi: \R^n\backslash\{0\} \rightarrow \R^n$ denote its gradient and $\nabla^2\phi: \R^n\backslash\{0\} \rightarrow S^2\R^n$ its Hessian, where $S^2\R^n$ stands for symmetric $2$ tensors.
Let
\begin{equation}\label{H}
  H = \frac{1}{2}h^2: \R^n \rightarrow (0,\infty).
\end{equation}
Under the assumptions above, the gradient map, 
$\nabla H = h\nabla h: \R^n\backslash\{o\} \rightarrow \R^n\backslash\{o\}$, is a $C^1$ diffeomorphism, and, by Lemma 5.5 in \cite{Yan10},
the following relations hold for any $\xi \in \R^n\backslash\{o\}$ and $x=\nabla H$:
\begin{align}
  h(\xi) &= \tilde{h}(\nabla H(\xi)) \label{support-bis}\\
  h(\xi)\nabla h(\xi) &= x \label{x}\\
  \xi &= h(\xi)\nabla \tilde{h}(\nabla H(\xi)) \label{xi}\\
  \det \nabla^2 H(\xi) &= h^{n+1}(\xi)\tilde{f}(\xi). \label{hessian}
\end{align}

The homogeneous contour integral of a function $\phi: \R^n\backslash\{0\} \rightarrow \R$,
with homogeneity degree $-n$, is defined as
\begin{equation}\label{integral}
  \oint \phi(x)\,dx = \int_{{\mathbb S}^{n-1}} \phi(u)\,d\mathcal{H}^{n-1}(u).
\end{equation}
The volume of $K$ is given by
\begin{equation}\label{volume}
  V(K)  = \frac{1}{n}\int_{{\mathbb S}^{n-1}} \tilde{h}(u)^{-n}\,du
   = \frac{1}{n}\oint \tilde{h}(x)^{-n}\,dx = \frac{1}{n}\oint h(\xi)f(\xi)\,d\xi.
\end{equation}

We also use the following integration by parts and change of variables lemmas. 

\begin{lemma}\label{parts} (Corollary 6.6, \cite{Yan10})
  Given a $C^1$ function $\phi: \R^n\backslash\{0\} \rightarrow \R$, homogeneous of degree $-n+1$, we have, for every 
  $j\in\{1,\dots,n\}$,
   \[
    \oint \partial_j\phi(x)\,dx = 0.
  \]
\end{lemma}

\begin{lemma}\label{change} (Corollary 6.8, \cite{Yan10})
  Given a $C^1$ function $\phi: \R^n\backslash\{o\} \rightarrow \R$ homogeneous of degree $-n$ and a $C^1$ diffeomorphism $\Phi: \R^n\backslash\{o\}  \rightarrow \R^n\backslash\{o\}$ homogeneous of degree $1$, we have
  \[
    \oint \phi(x)\,dx = \oint \phi(\Phi(\xi))\,\det\nabla\Phi(\xi)\,d\xi.
  \]
\end{lemma}

The following is the core result leading to Proposition~\ref{theodensp-n} where $\delta_{ij}$ stands for the usual Kronecker symbols
$\delta$.

\begin{lemma}
\label{fp}
  Given $1 \le i, j \le n$ and $p \ne 0$,
\begin{equation}
\int_{{\mathbb S}^{n-1}} u_ih^p(u)\partial_jf_p(u) \,du = 
-(n+p) V(K)\delta_{ij},
\end{equation}
where $f_p = h^{1-p}f$.
\end{lemma}

\begin{proof}
By \eqref{integral} and Lemma~\ref{parts},
\begin{equation}\label{1}
  \begin{split}
  \int_{{\mathbb S}^{n-1}} u_i\partial_j\tilde{h}(u) (\tilde{h}(u))^{-n-1}\,du
&= \oint x_i\partial_j \tilde{h}(x)(\tilde{h}(x))^{-n-1}\,dx\\
&= -\frac{1}{n}
\oint x_i\partial_j(\tilde{h}(x))^{-n}\,dx\\
&= \frac{1}{n}
\oint \partial_j(x_i)(\tilde{h}(x))^{-n}\,dx\\
&= \frac{1}{n}
\oint \delta_{ij}(\tilde{h}(x))^{-n}\,dx\\
&= V(K)\delta_j^i.
\end{split}
\end{equation}
On the other hand, using the change of variable $x = \nabla H(\xi)$, it follows by Lemma~\ref{change}, \eqref{x}, \eqref{xi}, \eqref{hessian}, Lemma~\ref{parts}, and \eqref{volume} that
\begin{equation}\label{2}
  \begin{split}
  \oint x_i\partial_j \tilde{h}(x)(\tilde{h}(x))^{-n-1}\,dx
  &=
    \oint (h(\xi)\partial_ih(\xi))\xi_j h^{-n-2}(\xi)\det\nabla^2H(\xi)\,d\xi\\
  &=
  \oint \partial_ih(\xi)\xi_j\tilde{f}(\xi)\,d\xi\\
  &=
  \oint (h^{p-1}\partial_ih)\xi_jh^{1-p}\tilde{f}\,d\xi\\
  &=
  \frac{1}{p}\oint \partial_i(h^p(\xi))\xi_j(h^{1-p}\tilde{f})\,d\xi\\
  &=
    -\frac{1}{p}\oint h^p(\xi)\partial_i(\xi_jh^{1-p}\tilde{f})\,d\xi\\
  &= -\frac{1}{p}\oint \delta_{ij}h(\xi)\tilde{f}(\xi) + \xi_jh^p(\xi)\partial_if_p(\xi)\,d\xi\\
  &= -\frac{n}{p}V(K)\delta_{ij} - \frac{1}{p}\oint \xi_jh^p(\xi)\partial_if_p(\xi)\,d\xi\\
  &= -\frac{n}{p}V(K)\delta_{ij} - \frac{1}{p}\int_{{\mathbb S}^{n-1}}
  u_jh^p(u)\partial_if_p(u)\,du
\end{split}
\end{equation}
The lemma now follows by \eqref{1} and \eqref{2}.
\end{proof}

Setting $p = -n$ in Lemma~\ref{fp}, we get Proposition~\ref{theodensp-n}.\\

\noindent{\bf Acknowledgement} We thank Bal\'azs Csik\'os, Gaoyong Zhang and Guangxian Zhu for helpful discussions.

\end{document}